\documentclass{article}

\title{Fixing the Kawarabayashi-Thomas-Wollan Flat Wall}
\author{Dan Arnon}

\usepackage{geometry}
\usepackage{amsmath}

\usepackage{amsthm}
\usepackage{amsopn}
\usepackage{parskip}   
\usepackage{lmodern}
\usepackage{amssymb}
\usepackage{tikz}
\usepackage{subcaption}
\usepackage{enumerate}
\usepackage{appendix}
\usepackage{mathtools}

\usetikzlibrary{patterns}

\theoremstyle{plain}
\newtheorem*{defnFlat}{Definition}

\newtheorem*{thm52}{Theorem 5.2}
\newtheorem{lem}{Lemma}
\newtheorem*{lem*}{Lemma}
\newtheorem*{lem51}{Lemma 5.1'}

\newtheorem{claim}{Claim}
\newtheorem{defn}{Definition}

\DeclareMathOperator{\tr}{\operatorname{tr}}
\DeclareMathOperator{\bd}{\operatorname{bd}}

\newcommand{\sphere}[1]{\mathbb{S}^{#1}}
\newcommand{\diskin}[1]{\Delta_{#1}^{\text{in}}}
\newcommand{\diskout}[1]{\Delta_{#1}^{\text{out}}}
\begin{document}

\maketitle

\begin{abstract}
Recent papers by Kawarabayashi, Thomas and Wollan (\cite{NewProof, QuicklyExcluding}) provide major improvements over Robertson and Seymour's original proof of the structure theorem for finite graphs that exclude a given graph (\cite{RSXVI}).
This structure theorem constitutes a central step in the proof of the Wagner Conjecture. The new papers provide a significant reduction of the size bounds in the theorem as well as providing a simpler, shorter and more accessible proof.
The first paper \cite{NewProof} gives a new proof of the Flat Wall Theorem, a central stepping stone to the proof of the structure theorem itself in \cite{QuicklyExcluding}.
More than that, the paper redefines an important notion, that of a {\em flat wall}. Unfortunately, this new notion is too strong. As a result, the new Flat Wall Theorem (Theorem 5.2 in \cite{NewProof}) is incorrect.
I give a counterexample in Appendix \ref{CounterExample}.
A follow-on lemma in the same paper (Lemma 6.1 in \cite{NewProof}, about the transitivity of flatness) is also incorrect, a fact that was noticed by Dimitrios Thilikos et al in \cite{MoreAccurate}.
However, those authors appear to have missed the main issue, which is Theorem 5.2. Nevertheless, their notion of a {\em tight rendition} is a crucial ingredient of the fix to the main problem as presented here.

This paper provides a weaker definition of the notion of a flat wall, provides a correction to the proof of the Flat Wall Theorem and a new proof of flatness transitivity.
The notion of a tight rendition as presented here differs a little from \cite{MoreAccurate} but is defined much more simply, and the notion of a proper cycle is introduced.
The notions of certificates and tilted walls in \cite{MoreAccurate} turn out of be unnecessary and transitivity is preserved in its original simplicity and generality.
Most importantly, it looks like the new weaker definition of flatness is all that is really necessary to carry through the structure theorem in \cite{QuicklyExcluding}.
\end{abstract}

\newpage
\tableofcontents

\section{Preliminaries}

This section is mostly a review of some basic concepts of Robertson and Seymour's graph minor theory, but it also includes an explanation of some terminology choices I made in this paper.

\subsection{Loops in $\sphere{2}$}
We start with some terminology and basic facts about subsets of the unit 2-sphere $\sphere{2}$ that are homeomorphic to the unit 1-sphere $\sphere{1}$.
We refer to them as {\em loops} though technically they are {\em simple} loops since they do not have self-intersections.

A loop $L \in \sphere{2}$ divides the sphere into two closed regions $\Delta_L^0$ and $\Delta_L^1$, both homeomorphic to a closed disk with intersection $\Delta_L^0 \cap \Delta_L^1 = L$.

Given two loops $L$ and $L'$, we say that $L$ and $L'$ are {\em non-crossing} if $L \subset \Delta_{L'}^0$ or $L \subset \Delta_{L'}^1$.
Let $L$ and $L'$ be non-crossing. Without loss of generality, we can assume $L \subset \Delta_{L'}^0$.
Since $\Delta_{L'}^0$ is simply connected, we have one region of $L$ that is contained in $\Delta_{L'}^0$. Say $\Delta_L^0 \subseteq \Delta_{L'}^0$.

The converse is also true. If $\Delta_L^{i} \subset \Delta_{L'}^j$ for some $i, j \in \{ 0, 1 \}$ then $L \subset \Delta_{L'}^j$ and so $L$ and $L'$ are non-crossing.
Since $\Delta_L^i \subset \Delta_{L'}^j$ implies $\Delta_{L'}^{1 - j} \subset \Delta_L^{1 - i}$, the non-crossing relationship is symmetric.

\begin{lem}\label{lem:NCPair4Options}
Let $L$ and $L'$ be non-crossing and let $i, i' \in \{ 0, 1 \}$. Then either:
\begin{align*}
\Delta_L^i \subseteq \Delta_{L'}^{i'} & \quad \text{or} \quad \Delta_{L'}^{i'} \subseteq \Delta_L^i \quad \text{or} \\
\Delta_L^i \cap \Delta_{L'}^{i'} \subseteq L \cap L' & \quad \text{or} \quad \Delta_L^i \cup \Delta_{L'}^{i'} = \sphere{2}
\end{align*}
\end{lem}

\begin{proof}
Since $L$ and $L'$ are non-crossing, we have some $j, k \in \{ 0, 1 \}$ such that $\Delta_L^j \subseteq \Delta_{L'}^k$.
It follows that $\Delta_{L'}^{1 - k} \subseteq \Delta_L^{1 - j}$.

If $j = i$ and $k = i'$ we are done due to the first inclusion, and if $j \ne i$ and $k \ne i'$ we are done due to the second inclusion. If $j = i$ and $k \ne i'$, then
\begin{alignat*}{2}
\Delta_L^i \cap \Delta_{L'}^{i'} & = \Delta_L^j \cap \Delta_{L'}^{1 - k} & & \subseteq \Delta_{L'}^k \cap \Delta_{L'}^{1 - k} = L' \\
\Delta_L^i \cap \Delta_{L'}^{i'} & = \Delta_L^j \cap \Delta_{L'}^{1 - k} & & \subseteq \Delta_L^j \cap \Delta_L^{1 - j} = L
\end{alignat*}
Finally if $j \ne i$ and $k = i'$ then
$$
\Delta_L^i \cup \Delta_{L'}^{i'} = \Delta_L^{1 - j} \cup \Delta_{L'}^k \supseteq \Delta_{L'}^{1 - k} \cup \Delta_{L'}^k = \sphere{2}
$$
\end{proof}

\subsection{Societies and Renditions}

The definitions of "painting" and "rendition" below are borrowed from \cite{NewProof} with a few corrections and superficial changes that make these notions easier to work with.

\subsubsection{Paintings and orientations}
Let $\mathbb{S}^2$ be the unit 2-sphere in $\mathbb{R}^3$.
A {\em bounded painting} in $\mathbb{S}^2$ is a quadruple $(\mathcal{N}, \bar{\mathcal{C}}, \bar{\star}, \tau)$ such that
\begin{itemize}
\item $\mathcal{N} \subset \mathbb{S}^2$ is a finite set of points, called the {\em nodes}  of the painting.
\item $\bar{\mathcal{C}}$ is a finite family of subsets of $\mathbb{S}^2$ each homeomorphic to a closed disk and $\bar{\star} \in \bar{\mathcal{C}}$ is a distinguished disk called the {\em external disk}.
\item Let $U = \bigcup\limits_{u \in \bar{\mathcal{C}}} u$. Then $\mathcal{N} \subset \bd(U)$ and the sets $\mathcal{C} \coloneqq \{ u \setminus \mathcal{N} \vert u \in \bar{\mathcal{C}} \}$
are the connected components of of $U \setminus \mathcal{N}$.
The members of $\mathcal{C}$ are called the {\em cells} of the painting and the distinguished cell $\star \coloneqq \bar{\star} \setminus \mathcal{N}$ is called the {\em external} cell.
We call all the other cells {\em internal}.
For each cell $c = u \setminus \mathcal{N}$, define $\tilde{c} = u \cap \mathcal{N}$.
\item  For every internal cell $c$, $\vert \tilde{c} \vert \le 3$.
As a result if $n_1, n_2$ are distinct nodes in $\tilde{c}$ then at least one of the two open segments of $\bd(c)$ created by $n_1, n_2$ is $\mathcal{N}$-free.
$\tau$ is a function (called the tie-breaker) that associates each such triple $c, n_1, n_2$ to an $\mathcal{N}$-free open segment $\tau(c, n_1, n_2)$ of the boundary of $c$.
\item The function $\tau$ is unoriented. For all internal $c$ and distinct $n_1, n_2 \in \tilde{c}$,
$$
\tau(c, n_1, n_2) = \tau(c, n_2, n_1)
$$
\end{itemize}

It follows immediately from the definition that the disks in $\bar{\mathcal{C}}$ are mutually almost disjoint and only touch at a finite number of nodes along their boundaries.
Pick an orientation of $\Delta \coloneqq \overline{\mathbb{S}^2 \setminus \star}$. The orientation establishes a notion of right and left when traversing a path in $\Delta$.

For a loop $\gamma$ in $\Delta$, define the {\em clockwise} orientation of $\gamma$ to be the direction of travel that keeps the interior region of $\gamma$ on the right.
In particular the boundaries $\bd(c)$ of cells $c \in C(\Gamma)$ have a distinguished clockwise orientation. These orientations induce a circular order on the sets $\pi(\tilde{c})$.

\begin{defn}
Let $\Delta \subset \mathbb{S}^2$ be an oriented disk and let $L \subset \Delta$ be a loop.
The loop $L$ divides $\mathbb{S}^2$ into two closed regions, both homeomorphic to a disk, which we denoted by $\Delta_L^0$ and $\Delta_L^1$.
$\Delta$ is simply connected and therefore exactly one of these region, denoted $\diskin{L}$, is a subset of $\Delta$.
The other region is denoted $\Delta_L^{\text{out}}$.
\end{defn}

If we endow the loop $L$ with a clockwise orientation, then the interior of $\diskin{L}$ will be on the right as we traverse $L$.
If a loop $L$ comes with a pre-determined arbitrary orientation, then one of its two regions will be on the right in the direction of travel.
We denote that region by $\Delta_{\overrightarrow{L}}$, and its complementary region by $\Delta_{\overleftarrow{L}}$.
It follows that if $L$ is oriented, then its orientation is clockwise if and only if $D_{\overrightarrow{L}} = \diskin{L}$.

Given an oriented simple loop $L \subset \Delta$ and two distinct points $x, y \in L$,
the segment $L[x, y]$ is defined to be the segment of $L$ that one would traverse when traveling in $L$ from $x$ to $y$ in the direction given by the orientation.
It follows that $L = L[x, y] \cup L[y, x]$.

\begin{lem}
\label{lem:twononcrossingloops}
Let $\Delta \subset \mathbb{S}^2$ be an oriented disk, 
and let $L_0, L_1  \subset \Delta$ be two non-crossing loops.
let $k \in \{ 2, 3 \}$ and let $n_1, \dots, n_k \in L_0$ be distinct points listed in $L_0$-clockwise order.
Let $X$ be a union of  some of the open segments $L_0(n_1, n_2), \dots, L_0(n_k, n_1)$.

Assume that $\diskin{L_0} \cap \diskin{L_1} = \{ n_1, \dots, n_k \} \cup X$ (see for example Figures \ref{fig:3noncross0} and  \ref{fig:2noncross1}.)
For $1 \le m \le k$ and $p = 1 +  (m \mod k)$ such that $L_0(n_m, n_p) \not\subseteq X$, define $Z_m$ to be the loop
$$
Z_m = L_0[n_m, n_p]  L_1[n_p, n_m] 
$$

Then
\begin{itemize}
\item each $Z_m$ is a simple loop that is non-crossing relative to $L_0$ and $L_1$.
\item some $\diskin{Z_m}$ contains both $\diskin{L_0}$ and $\diskin{L_1}$.
\end{itemize}
\end{lem}

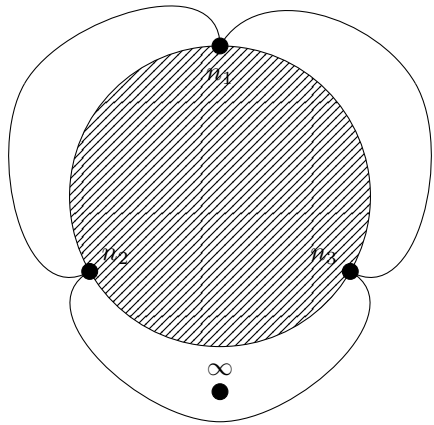
\begin{figure}[htb]
\centering
\begin{tikzpicture}
\draw[pattern = north east lines] (0, 0) circle (2);

\filldraw (1.732, -1) circle (3 pt);
\node at (1.386, -0.8) {$n_3$};
\filldraw (-1.732, -1) circle (3 pt);
\node at (-1.386, -0.8) {$n_2$};
\filldraw (0, 2) circle (3 pt);
\node at (0, 1.6) {$n_1$};

\draw (-1.732, -1) to[in = 180, out = 210] (0, -3);
\draw (0, -3) to[in = -30, out = 0] (1.732, -1);

\draw (1.732, -1) to[in = -60, out = -30] (2.598, 1.5);
\draw (2.598, 1.5) to[in = 60, out = 120] (0, 2);

\draw (0, 2) to[in = 60, out = 90] (-2.598, 1.5);
\draw (-2.598, 1.5) to[in = -150, out = -120] (-1.732, -1);

\filldraw (0, -2.6) circle (3 pt);
\node at (0, -2.3) {$\infty$};

\end{tikzpicture}
\caption{$k = 3$, $X = \emptyset$. $\diskin{L_0}$ is lined. $\diskin{L_1}$ is the outer region. $\infty$ is outside $\Delta$.}
\label{fig:3noncross0}
\end{figure}

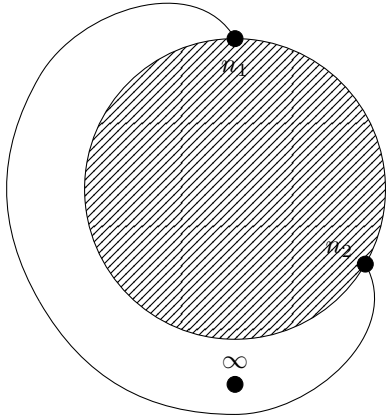
\begin{figure}[htb]
\centering
\begin{tikzpicture}
\draw[pattern = north east lines] (0, 0) circle (2);

\filldraw (1.732, -1) circle (3 pt);
\node at (1.386, -0.8) {$n_2$};
\filldraw (0, 2) circle (3 pt);
\node at (0, 1.6) {$n_1$};

\draw (-2.598, -1.5) to[in = 180, out = -60] (0, -3);
\draw (0, -3) to[in = -60, out = 0] (1.732, -1);

\draw (-2.598, -1.5) to[in = -120, out = 120] (-2.598, 1.5);
\draw (-2.598, 1.5) to[in = 120, out = 60] (0, 2);

\filldraw (0, -2.6) circle (3 pt);
\node at (0, -2.3) {$\infty$};

\end{tikzpicture}
\caption{$k = 2$, $X = L_0(n_1, n_2)$. $\diskin{L_0}$ is lined. $\diskin{L_1}$ is the outer region. $\infty$ is outside $\Delta$.}
\label{fig:2noncross1}
\end{figure}

\begin{proof}
Throughout the proof we use $m$ to range over $\{ 1, \dots, k \}$ and $i$ to range over $\{ 0, 1 \}$.
We also use the notation $p \coloneqq 1 + (m \mod k)$ for the modular successor of $m$.
$Z_m$ is defined as the union of two simple paths. The intersection of these paths is
$$
L_0[n_m, n_p] \cap L_1[n_p, n_m] \subseteq L_0[n_m, n_p] \cap ( \{ n_1, \dots, n_k \} \cup X ) = \{ n_m, n_p \}
$$
proving that $Z_m$ is a simple loop.

The loop $Z_m$ is composed of a segment of $L_i$ (which resides in $\diskout{L_i}$ by definition),
and a segment of $L_{1 - i}$ which resides in $\diskout{L_i}$ by assumption. Therefore $Z_m \subset \diskout{L_i}$,
proving that $Z_m$ and $L_i$ are non-crossing.

We prove the second claim through the following steps:
\begin{enumerate}[Step 1:]
\item For all $m$ and $i$, neither $\diskin{Z_m} \cup \diskin{L_i} = \sphere{2}$ nor $\diskin{Z_m} \subseteq \diskin{L_i}$ hold. \\
The relation $\diskin{Z_m} \cup \diskin{L_i} = \sphere{2}$ is not possible since by definition $\diskin{Z_m} \cup \diskin{L_i} \subseteq \Delta \subsetneq \sphere{2}$.
Assume that $\diskin{Z_m} \subseteq \diskin{L_i}$. This implies
$$
\diskin{Z_m} \cap \diskin{L_{1 - i}} \subseteq \diskin{L_i} \cap \diskin{L_{1 - i}} = \{ n_1, \dots, n_k \} \cup X
$$
We know that by definition $\diskin{Z_m} \cap \diskin{L_0}$ contains the open segment $L_0(n_m, n_p)$ which is not in $X$,
so the case $i = 1$ is not possible.

$\diskin{Z_m} \cap \diskin{L_1}$ contains the open segment $L_1(n_p, n_m)$. This implies $L_1[n_p, n_m] \subset \bar{X}$ so it is a segment of $L_0$.
The equality $L_1(n_p, n_m) = L_0(n_m, n_p)$ is not possible, because by definition $L_0(n_m, n_p) \not\subseteq X$.
The equality $L_1(n_p, n_m) = L_0(n_p, n_m)$ is not possible
because the interior of both $\diskin{L_0}$ and $\diskin{L_1}$ is found on the right of the segment as we travel from $n_p$ to $n_m$, 
implying that the intersection of the two disks has an interior, contrary to our assumption.

\item There are $m$ and $i$ such that $\diskin{Z_m} \supseteq \diskin{L_i}$. \\
Assume this is not the case. By Lemma \ref{lem:NCPair4Options}, there are four possible relations between $\diskin{Z_m}$ and $\diskin{L_i}$.
We have already eliminated two of them.
By assumption, the relation $\diskin{L_i} \subseteq \diskin{Z_m}$ does not occur either.
As a result, the last relation, $\diskin{Z_m} \cap \diskin{L_i} = Z_m \cap L_i$ must be true for all possible values of $m$ and $i$.
Let
$$
R = \diskin{L_0} \cup \diskin{L_1} \cup \bigcup\{ \diskin{Z_m} \}
$$

The boundary $\bd(R)$ is a subset of $L_0 \cup L_1$.
We will show that $\bd(R) = \emptyset$. Recall that $m, p$ always indicate consecutive indices in the clockwise circular order on $L_0$.
Let $x \in L_0 \cup L_1 \setminus \{ n_1, \dots, n_k \}$.

Suppose that $x \in L_0(n_m, n_p)$.
If $L_0(n_m, n_p) \subseteq X$ then it is a common boundary segment of $\diskin{L_0}$ and $\diskin{L_1}$.
Since $\diskin{L_0} \cap \diskin{L_1}$ is nowhere dense by assumption, these disks do not share interior points and therefore their interiors are on opposite sides of $L_0(n_m, n_p)$.
It follows that $x$ is not a boundary point of $R$ in that case.
If $L_0(n_m, n_p) \not\subseteq X$ then it is a common boundary segment of $\diskin{L_0}$ and $\diskin{Z_m}$ and by a similar argument $x$ is not a boundary point of $R$.

Suppose that $x \in L_1(n_p, n_m)$.
If $L_0(n_m, n_p) \not\subseteq X$ then $L_1(n_p, n_m)$ is a common boundary segment of $\diskin{L_1}$ and $\diskin{Z_m}$ and by the same argument as before, $x$ is not a boundary point of $R$.
If $L_0(n_m, n_p) \subseteq X$ then it is a segment of $L_1$ as well, and it must be equal to either $L_1(n_p, n_m)$ or $L_1(n_m, n_p)$.
In the former case $x \in L_0(n_m, n_p)$ and we have already shown that $x$ is not a boundary point of $R$.
In the latter case, when $L_0(n_m, n_p)$ is traversed from $n_m$ to $n_p$, the interior of $\diskin{L_0}$ is on the right by definition,
and since $L_0(n_m, n_p) = L_1(n_m, n_p)$ the interior of $\diskin{L_1}$ is on the right as well, contradicting our assumption that these two disks do not share interior points.

It follows that $\bd(R) \subseteq \{ n_1, \dots, n_k \}$ and is therefore discrete. Since $R$ is a closed set, the boundary must be the empty set.
It follows that $R = \sphere{2}$, contradicting the fact that $R \subseteq \Delta$.

\item If $\diskin{Z_m}$ contains $\diskin{L_i}$ then it contains $\diskin{L_{1 - i}}$. \\
Assume $i = 0$ and traverse the loop $Z_m$ starting with the segment $L_0[n_m, n_p]$ in the direction from $n_m$ to $n_p$.
We encounter the interior of $\diskin{L_0}$ on the right since we are traversing the segment in the clockwise $L_0$-direction by definition.
Since $\diskin{Z_m}$ contains $\diskin{L_0}$, we must encounter the interior of $\diskin{Z_m}$ on the right as well.
In other words, we are traversing $Z_m$ in its own clockwise direction.
As we continue to traverse $Z_m$ in the same direction, we traverse $L_1[n_p, n_m]$ from $n_p$ to $n_m$.
The interior of $Z_m$ is still on the right, and so is the interior of $L_1$,
since we are traversing $L_1[n_p, n_m]$ in the $L_1$-clockwise order by definition.
As a result $\diskin{Z_m}$ and $\diskin{L_1}$ share an interior.
As we have seen, this implies that $\diskin{Z_m} \supseteq \Delta_{L_1}$.
A similar argument works when $i = 1$, we just need to start traversing $Z_m$ along $L_1[n_p, n_m]$ from $n_p$ to $n_m$.
This concludes the proof.
\end{enumerate}
\end{proof}

\subsubsection{Societies and their renditions}
Let $G$ be a graph and $C \subseteq V(G)$ a set of vertices of $G$ which we endow with a circular order. The pair $(G, C)$ is called a {\em society}.
Societies arise naturally when we try to draw graphs on two dimensional surfaces with a connected boundary (i.e. surfaces from which the interior of a closed disk has been removed),
because the boundary induces a natural circular order on the vertices that are drawn along it.

Let $(G, C)$ be a society and $\Delta \subset \mathbb{S}^2$ be a closed oriented disk.
A $C$-rendition of $G$ in $\Delta$ is a triple $(\Gamma, \sigma, \pi)$ such that
\begin{itemize}
\item $\Gamma$ is a bounded painting in $\mathbb{S}^2$ with $\bar{\star}(\Gamma) = \overline{\mathbb{S}^2 \setminus \Delta}$.
As we saw, the orientation of $\Delta$ induces a circular order on each set $\tilde{c}$ in $C(\Gamma)$.
\item $\pi : \mathcal{N} \hookrightarrow V(G)$ is a 1:1 association of nodes to vertices of $G$, and $\pi(\tilde{\star}) = C$ as cyclically ordered sets.
\item $\sigma$ is a function assigning a subgraph $\sigma(c) \subseteq G$ , called the {\em flap} of $c$, to each cell $c \in \mathcal{C}$,
such that
  \begin{itemize}
  \item $G = \bigcup_{c \in C} \sigma(c) \cup \pi(\mathcal{N})$
  \item for any node $c$, $\sigma(c) \cap \pi(\mathcal{N}) = \pi(\tilde{c})$, and $\sigma(\star) = \pi(\tilde{\star}) = C$.
  \item if $c_1 \ne c_2$ then $\sigma(c_1) \cap \sigma(c_2) = \pi(\tilde{c_1} \cap \tilde{c_2})$.
  \end{itemize}
\end{itemize}

A society $(G, C)$ is called {\em rural} if there is a $C$-rendition of $G$ in a disk.
The notion of a rural society is originally due to Robertson and Seymour in \cite{RSIX}, but the definition here (taken from Section 2.1 of \cite{QuicklyExcluding}) is clearer.

\subsection{Graph terminology vs. topological terminology}
Like tree decompositions, the power of renditions comes from the interplay they engender between the structure of the underlying graph and the geometry of the rendition.
To minimize confusion, I try to use separate terminologies for graph terms and the rendition-related topological terms.

\begin{itemize}
\item A subset of a disk homeomorphic to the unit circle $\mathbb{S}^1$ will be referred to as a {\em loop} (the usual term is {\em simple loop}, but all the loops here will be simple).
A cycle in a graph will be referred to as a {\em cycle}.
\item Loops in an oriented disk will be {\em oriented} and have {\em (clockwise) orientations}. Directed cycles in a graph will be {\em directed} and have {\em (clockwise) directions}.
\item Graphs have {\em vertices} and renditions have {\em nodes}.
When there is no risk of confusion, I will use the term {\em node} to refer to a graph vertex that is associated to a node through the $\pi$-mapping,
and often ignore $\pi$ completely and just assume that $\mathcal{N} \subseteq V(G)$ and $\pi$ is the identity.
\item If $H$ is a subgraph of $G$, the nodes of $H$ are denoted $N(H)$. So
$$
N(H) = V(H) \cap \pi(\mathcal{N}) = V(H) \cap \mathcal{N}
$$
If $X$ is a subset of $\Delta$, the nodes of $X$ are denoted $N(X)$. So $N(X) = X \cap \mathcal{N}$.
\item A subset of a disk $\Delta$ homeomorphic to the closed unit interval is usually referred to as a path in $\Delta$ , but here will be referred to as a {\em segment}. The term {\em path} will be reserved to paths in a graph.
The interior of a segment in $\Delta$ will be referred to as an {\em open} segment.
\item By abuse of notation, we will say that an edge (or vertex) of $G$ resides in a cell if it belongs to the flap of that cell.
\end{itemize}

\subsection{Tracks and proper cycles}

Throughout this subsection, $(G, C)$ is a rural society and $\rho = (\Gamma, \sigma, \pi)$ is a $C$-rendition of $G$ in an oriented disk $\Delta \subset \sphere{2}$.
Throughout, $\mathcal{N}$, $\mathcal{C}$ and $\tau$ refer to $\mathcal{N}(\Gamma)$, $\mathcal{C}(\Gamma)$ and $\tau(\Gamma)$. 
Most of the definitions here are taken from \cite{NewProof} and are rephrased here for completeness.
I tweaked the notion of {\em track} slightly to make it easier to work with.

\begin{defn}
  A path $P$ in $G$ is called {\em grounded} (relative to $\rho$) if both ends of $P$ are nodes of $\rho$.
  A cycle $D \subseteq G$ is {\em grounded} if there are two edges of $D$ that reside in different cells of $\rho$.
  A directed grounded path $Q$ is called {\em atomic} if $\vert E(Q) \vert \ge 1$ and none of its internal vertices are nodes of $\rho$.
  An atomic path $Q$ is called {\em trivial} if it contains no internal vertices, i.e. $\vert E(Q) \vert = 1$.
  \end{defn} 
  
  \begin{lem}
  An atomic path must reside entirely inside $\sigma(c)$ for some internal cell $c$.
  \end{lem}
  
  \begin{proof}
  Let $Q$ be an atomic path and let $v$ be an internal vertex of $Q$. Then $v$ is not a node and $v \in V(\sigma(c))$ for a unique cell $c$.
  There are two edges of $Q$ that end in $v$ and both must be in $\sigma(c)$.
  Therefore every adjacent pair of $Q$-edges share a cell and therefore all the edges of $Q$ share a cell.
  Since $Q$ has at least one edge, the shared cell must be internal.
  \end{proof}
  
  \begin{defn}
  Let $Q$ be an atomic path. We call the shared cell of the edges of $Q$ the {\em home} of $Q$ or $h(Q)$.
  \end{defn}
  
  \begin{defn}
  It is not hard to see that a directed grounded path $P$ with $k$ nodes ($k \ge 1$) can be written uniquely as a concatenation of $k-1$ atomic paths
  (in the case $k = 1$ this empty concatenation should be interpreted as the node that consists the whole path.)
  These atomic paths are called the {\em factors} of $P$. If $m, n$ are consecutive nodes of $P$ in the given direction then the factor of $P$ connecting $m$ to $n$ is the directed subpath $P[m, n]$.
  
  A directed grounded cycle $D$ with $k$ nodes ($k \ge 2$) can be written uniquely (up to rotation) as a circular concatenation of $k$ atomic paths.
  If $m, n$ are consecutive nodes of $D$ in the given direction then the factor of $D$ connecting $m$ to $n$ is the subpath  $D[m, n]$ that proceeds from $m$ to $n$ in the given direction.
  
  We call these representations the {\em atomic decomposition} of the oriented path (or cycle). The atomic components of the decompositions are called the {\em factors} of the path (or cycle).
  \end{defn}
  
  The homes of consecutive atomic paths in an atomic decomposition do not have to be distinct.
  A cell can appear as a home either once or twice in an atomic decomposition, and any pair of factors with the same home must be (cyclically) adjacent.
  It is clear from the definition of groundedness that if a directed grounded cycle $D$ has an atomic decomposition of length two, $D = Q_1 Q_2$, then $h(Q_1) \ne h(Q_2)$.

  \begin{defn}
  Given an atomic path $Q$ with ends $n_1 \rightarrow n_2$, define $s_Q = \tau(h(Q), n_1, n_2)$.
  Then $s_Q$ is an $\mathcal{N}$-free open segment of the boundary of $h(Q)$ that connects the nodes at the two ends of $Q$.
  
  For a simple grounded path $P$, choose a direction for $P$ and let $n_0, \dots n_{k - 1}$ be the nodes of $P$, written in order, where $k \ge 1$. Then the {\em track} of $P$ is the set
  $$
  \tr(P) = ( \bigcup\limits_{i = 1}^{k-1} s_{P[n_{i-1}, n_i]} ) \cup \{ n_0, \dots, n_{k - 1} \}
  $$
  For a simple grounded cycle $D$, choose a direction of $D$ and let $n_0, \dots n_{k - 1}$ be the nodes of $D$, written in circular order, where $k \ge 2$. Then the {\em track} of $D$ is the set
  $$
  \tr(P) = ( \bigcup\limits_{i = 1}^k s_{D[n_{i-1}, n_{(i \mod k)}]} ) \cup \{ n_0, \dots, n_{k - 1} \}
  $$
  Because $\tau$ is unoriented, the definition of track is independent of the choice of direction of the path (or cycle).
  The track of a cycle is also independent of the choice of the starting node $n_0$.
  In the case of a grounded path $P$, the track of $P$ is a segment in $\Delta$ connecting the ends of $P$. In the case of a grounded cycle $D$, the track of $D$ is a loop in $\Delta$.
  In either case the track passes through all the nodes of $P$ (or $D$),  and no other nodes.
  
  While the track does not depend on the chosen direction of the path or cycle, each direction of the path or cycle induces an orientation on the track, and vice versa.
  The track of a grounded cycle $D$ has a clockwise orientation induced by the orientation of $\Delta$. This orientation, in turn, induces a direction on $D$ itself. We call that direction the {\em clockwise} direction of $D$.

  For a cycle $D$, $\tr(D)$ is a loop in $\Delta$. We write $\diskin{D} \coloneqq \diskin{\tr(D)}$ and $\diskout{D} \coloneqq \diskout{\tr(D)}$. 
  
\end{defn}

\begin{lem}
Let $P$ and $R$ be two grounded paths (or cycles) in $G$. Then $P$ and $R$ intersect if and only if their tracks intersect.
If the intersection $P \cap R$ is a path, then the intersection of their tracks is a segment or a loop.
\end{lem}

\begin{proof}
Choose directions for $P$ and $R$.
Suppose $P$ and $R$ intersect and let $v \in V(P) \cap V(R)$. If $v$ is a node then it is on the intersection of the tracks.
Otherwise $v$ belongs to a unique cell $c$. Let $Q_P$ and $Q_R$ be the unique atomic factors of $P$ and $R$ respectively with $v \in V(Q_P) \cap V(Q_R)$.
Then $h(Q_P) = h(Q_R) = c$ and therefore $P$ and $R$ have at least two nodes each in the
boundary of $c$. Since $\vert \tilde{c} \vert \le 3$ it follows that at least one of these nodes is common to both and belongs to the intersection of the tracks.
Conversely assume that the tracks of $P$ and $R$ share a node. This node is by definition a vertex of both, so the paths intersect.

Suppose $X = P \cap R$ is a path. By the preceding argument, $X$ contains at least one node.
If $\tilde{X}$ is the longest grounded subpath of $X$, it is not hard to see that the intersection of the tracks of $P$ and $R$ is usually the track of $\tilde{X}$, except in the case where $P$ and $R$ are both cycles and share the same track.
\end{proof}

Finally we need a few more definitions.
When consecutive factors of an directed cycle $D$ share the same home $c$, their tracks also share $c$. In such a case $\tilde{c}$ must have three nodes, and all of them are nodes of $D$,
and there are two consecutive segments of $\tr(D)$ along the boundary of $c$.
There are other complex ways for a cell of degree $3$ to interact with $\tr(D)$.
If $\vert \tilde{c} \vert = 3$ and $\tilde{c} \subseteq N(D)$ then the track of $D$ may have two, one or zero segments along the boundary of $c$.
We are interested in cycles $D$ where the interaction of $\tr(D)$ with the homes of its factors is particularly simple. These {\em proper} cycles are of particular importance for the arguments presented in this paper.

  \begin{defn}
  Let $c$ be an internal cell of $\rho$. The {\em degree} of $c$ is the number of nodes on the boundary of $c$.
  If $D$ is a grounded cycle, then $c$ is {\em internal relative to $D$} if $c \subseteq \diskin{D}$, and otherwise it is {\em external relative to $D$}.
  If $c$ is external relative to $D$ and at least on of the factors of $D$ has $c$ as its home, then we say that $c$ is a {\em border} cell of $D$.
  \end{defn}
  
  \begin{defn}
  A cycle $D$ is called {\em proper} if each border cell $c$ of $D$ has degree $3$, and exactly two of the nodes of $c$ are nodes of $D$. It follows immediately that the third node is in the interior of $\Delta_D^{\text{out}}$.
  \end{defn}

\section{Fixing a wall}

\subsection{Fixing the definition of flatness}
In \cite{NewProof}, the definition of the term {\em flat wall} is too restrictive. It ignores the possibility that the induced graph $G[A \cap B]$ of the torso $A \cap B$ of the separation $(A, B)$ may possess crossing edges that can prevent
the typical wall from being flat. This is exactly how the counterexample to the version of the Flat Wall Theorem in \cite{NewProof} is constructed. See Appendix \ref{CounterExample} for a description of the counterexample.

The following weaker definition seems to be sufficient for the purpose of \cite{NewProof, QuicklyExcluding}. We start with a definition that will help us reason about pegs.

\begin{defn}
  Let $W_e$ be a an elementary wall. Let $D$ be the boundary of $W_e$. A subpath $P$ of $D$ is called a {\em peg interval of $W_e$} if $V(P) > 2$; both ends of $P$ are degree $3$ vertices of $W_e$;
  and all of the internal vertices of $P$ are pegs, i.e. degree $2$ vertices of $W_e$.
  If $W$ is a wall with boundary $D$ that is obtained by an edge subdivision of an elementary wall $W_e$, then a subpath $P$ of $W$ is a {\em peg interval of $W$} if it is a subdivision of a peg interval of $W_e$.
  Notice that this definition does not depend on the choice of $W_e$.
\end{defn}

\begin{defn}\label{def:flatwall}
  Let $G$ be a graph and $W \subset G$ a wall with boundary $D$. We say that $W$ is {\em flat} in $G$ if there is a separation $(A, B)$ of $G$ and a vertex set $\Omega \subseteq A \cap B$ such that
  \begin{enumerate}
  \item $V(W) \subset V(B)$
  \item $A \cap B \subset V(D)$
  \item $\Omega$ intersects the interior of each peg interval of $W$.
  \item Endow $\Omega$ with a circular order induced from $D$. Then $(G[B], \Omega)$ is a rural society, as defined in Section 2.1 of \cite{QuicklyExcluding}.
  \end{enumerate}
\end{defn}

This definition is weaker than the overly strong definition in \cite{NewProof} in two respects. It does not require $\Omega = A \cap B$,
and it does not require that there be a choice of an elementary wall for $W$ such that every peg of it is in $\Omega$. The only requirement is for one peg in every peg interval to be present in $\Omega$. 
An elementary wall has peg intervals that contain two or even three distinct pegs.
Notice however that the {\em corners} of $W_e$ reside in distinct peg intervals, so $W_e$ can be chosen such that every corner of $W$ is in $\Omega$.

\subsection{Fixing Lemma 5.1 in \cite{NewProof}}

Lemma 5.1 is a technical lemma, used in \cite{NewProof} to prove the Flat Wall Theorem 5.2. The original lemma needs to be restated due to the problematic flat wall definition, and in addition it needs to be proved more carefully because of subtleties that went unnoticed in the original proof.

We start with defining tight renditions and proving their basic properties.

\begin{defn}
Let $\rho$ be a rendition of a society in a disk.
The {\em degree} of $\rho$ is the sum of the degrees of all the cells of $\rho$.
\end{defn}

\begin{defn}
Let $\rho$ be a rendition of a society in a disk, and let $c$ be a cell of $\rho$.
We call $c$ {\em empty} if the flap $\sigma(c)$ is an edgeless graph. 
Clearly the number of non-empty cells in $\rho$ is bounded by $\vert E(G) \vert$.
\end{defn}


%
Since we have a global bound on the number of non-empty cells of a $C$-rendition of $G$ in a disk, we can define the following.
\begin{defn}
Let $(G, C)$ be a rural society. A {\em maximal $C$-rendition} of $G$ in a disk is a $C$-rendition of $G$ in a disk with the maximal possible number of non-empty cells.
A maximal rendition is {\em tight} if it has a minimum degree among all maximal renditions.
\end{defn}

\begin{lem}\label{TightRenditionProps}
Let $(G, C)$ be a rural society, and let $\rho = (\Gamma, \sigma, \pi)$ be a tight $C$-rendition of $G$ in a disk. Then $\rho$ has the following properties:
\begin{enumerate}
\item\label{PropHomeAlone} If $Q$ is a trivial atomic path, it is home alone, i.e. $\sigma(h(Q)) = Q$. In particular the cell $h(Q)$ has degree two.
\item\label{PropThreeConnected} Let $c$ be a cell of $\rho$. If $c$ is non-empty and the nodes of $c$ are not  isolated in $G$, then the nodes of $c$ belong to a single connected component of $\sigma(c)$.
\item\label{PropTwoConnected} Let $c$ be a cell of $\rho$. If $c$ has degree $3$ and the nodes of $c$ belong to a single connected component of $\sigma(c)$ then it is not possible for one of the nodes of $c$ to separate the other two in $\sigma(c)$.
\end{enumerate}
\end{lem}

\begin{proof}
To prove claim \ref{PropHomeAlone}, assume that $Q \subsetneq \sigma(h(Q))$.
Create a new $C$-rendition $\rho' = (\Gamma', \sigma', \pi')$ of $G$, starting with $\rho' \coloneqq \rho$ and then modifying it as follows.
Create a new empty cell $c$ right next to $h(Q)$, along the track of of $Q$ and with the nodes of $Q$ serving as the only nodes of $c$.
Let $e$ be the lone edge of $Q$. Redefine $\sigma'(h(Q)) = \sigma(Q) \setminus e$ and set $\sigma'(c) = Q$.
Neither $h(Q)$ nor $c$ are empty in the modified rendition, so it has one more non-empty cell than $\rho$, contradicting its maximality.

To prove claim \ref{PropThreeConnected}, first suppose that one of the nodes of $c$, say $m$, is isolated in $\sigma(c)$.
Since $m$ is not isolated in $G$, it must be a node of some additional cell $c'$.  We can extract $m$ from $c$ by trimming the boundary of $c$ around $m$ and re-defining $\sigma'(c) = \sigma(c) \setminus m$.
This change does not make $c$ empty because $\sigma(c)$ has an edge, and $m$ is not an end of any edge of $\sigma(c)$, so the new rendition is still maximal.
We did, however, reduce the degree of the rendition, violating the tightness of $\rho$. So this case is not possible, and so the nodes of $c$ are not isolated in $\sigma(c)$.

Suppose $\sigma(c) = M \sqcup P$ where $M$ and $P$ are disjoint and $c$ has nodes $m \in V(M)$, $p \in V(P)$.
Since $m$ and $p$ are not isolated in $\sigma(c)$, both $M$ and $P$ have edges, and we can replace the cell $c$ with a pair of disjoint cells $c_M$ and $c_P$, with $\tilde{c}_M = N(M)$,
$\sigma(c_M) = M$, $\tilde{c}_P = N(P)$, and $\sigma(c_P) = P$.
The resulting rendition violates the maximality of $\rho$, since neither $c_M$ nor $c_P$ is empty.

To prove claim \ref{PropTwoConnected}, let $\tilde{c} = \{ m, n, p \}$, and assume that $n$ separates nodes $m$ and $p$ in $\sigma(c)$.
Then there is a separation $(M, P)$ of $\sigma(c)$ with $m \in V(M)$, $p \in V(P)$ and  $M \cap P = \{ n \}$.
Modify $\rho$ by replacing the cell $c$ by two cells $c_m$ and $c_p$ of degree $2$ each, such that $\tilde{c}_m = \{ m, n \}$ and $\tilde{c}_p = \{ n, p \}$.
Define $\sigma(c_m) = M$ and $\sigma(c_p) = P$. Since $m$ and $p$ are both connected to $n$ in $\sigma(c)$, both $c_m$ and $c_p$ are non-empty and the maximality of $\rho$ is violated.
\end{proof}

Finally, here is the replacement lemma for Lemma 5.1 of \cite{NewProof}.
\begin{lem51}
  Let $(G, C)$ be a rural society with $\vert C \vert \ge 4$.
  Let $W \subset G$ be a subgraph and $D \subset W$ be a directed cycle. Assume the following:
  \begin{enumerate}
  \item $W' = W \setminus V(D)$ is connected
  \item There are four simple paths $P_1, \dots, P_4$ from $C$ to $W'$ that are vertex-disjoint (with the possible exception of their $W'$ ends) and
  for each $1 \le i \le 4$, the intersection $P_i \cap D$ is a non-empty path. Let $x_i, y_i \in V(G)$ be the $C$-end and $W'$-end of $P_i$, respectively.
  \end{enumerate}
  Then $D$ is grounded relative to any $C$-rendition of $G$ in a disk,
  and one can choose a specific $C$-rendition $\rho$ of $G$ in an oriented disk $\Delta$ and a proper $\rho$-grounded cycle $E \subset G$ with the following properties:
  \begin{enumerate}
  \item\label{Dclockwise} the given direction of $D$ agrees with its induced clockwise direction.
  \item\label{VEsubVD} $N(E) \subseteq N(D)$, and the circular orders of $N(E)$ induced by the clockwise directions of $E$ and $D$ agree.
  \item\label{Efactors} For any clockwise consecutive nodes $m \rightarrow n$ of $E$, if $m \rightarrow n$ are also clockwise consecutive in $D$, then $E[m, n] = D[m, n]$.
  \item\label{DeltaEsupset} $\diskin{E} \supseteq \diskin{D}$
  \item Let $(A, B)$ be the following separation of $G$:
    \begin{itemize}
    \item $A = N(\Delta^{\text{out}}_E) \cup ( \bigcup\limits_{c \subseteq \Delta_E^{\text{out}}} \sigma(c) )$
    \item $B = (V(D) \cap V(E)) \cup ( \bigcup\limits_{c \subseteq \diskin{E}} \sigma(c) )$
    \end{itemize}
    Let $P \subset V(D) \setminus \mathcal{N}(\rho)$ be a set with the following properties:
    \begin{itemize}
    \item If $c$ is a border cell of both $D$ and $E$ then $\vert P \cap \sigma(c) \vert \le 1$.
    \item For any other cell $c$, $P \cap \sigma(c) = \emptyset$.
    \end{itemize}
    Then $V(W) \subseteq V(B)$, and if we define $\Omega = N(E) \cup P$, then $\Omega \subseteq A \cap B \subseteq V(D)$ and the society $(G[B], \Omega)$ is rural
    where the circular order on $\Omega$ is induced by the clockwise direction of $D$.
  \end{enumerate}
\end{lem51}
\begin{proof}
We proceed through a series of steps. We will choose $E$ and $\rho$ after completing a few necessary steps.
 \begin{enumerate}[Step 1:]
 \item $D$ is grounded relative to any $C$-rendition $\rho'$ of $G$ in a disk. \\
 Suppose that $D$ is not $\rho'$-grounded. Then all the edges of $D$ share a common cell $c$ of $\rho'$. Each path $P_i$ meets a first vertex $d_i$ of $D$ as it proceeds from $C$ to $W'$.
 The vertices $d_i$ must be nodes of $c$, and they must be distinct because they are not at the $W'$ end of their respective paths. Therefore $\vert \tilde{c} \vert > 3$ which is impossible since $c \ne \star$.
 \item In any $C$-rendition $\rho'$ of $G$ in a disk, $W'$ contains a node. \\
 This is proved using the same argument as the previous case.
 If $W'$ does not contain a node, then it follows from its connectedness that $W'$ must reside inside a single cell $c$ and the paths $P_i$ must each meet its boundary at a node.
 Since we assumed that these nodes are not in $W'$, they must be distinct, so $\vert \tilde{c} \vert > 3$. The last edge of each path $P_i$ must be in $\sigma(c)$ so $c \ne \star$, a contradiction.

 \item\label{step:ChooseRhoE} Choosing $\rho$ and $E$ \\
Let $\Delta$ be a disk. Look at the set of all pairs $(\rho, E)$ such that
\begin{itemize}
\item $\rho$ is a tight $C$-rendition of $G$ in $\Delta$, with $\Delta$ oriented such that the induced clockwise direction of $D$ agrees with its given direction.
\item $E$ is a $\rho$-grounded cycle that meets the required properties \ref{VEsubVD}, \ref{Efactors} and \ref{DeltaEsupset}
but is not necessarily proper.
(for property \ref{VEsubVD}, notice that the direction of $D$ is well defined because it is automatically $\rho$-grounded as we have already established).
\end{itemize} 
This set is not empty because a tight $\rho$ exists (since $(G, C)$ is rural), and for such $\rho$ the pair $(\rho, D)$ meets the criteria.
Among all possible choices of $(\rho, E)$ in the set, choose one where the graph $B$ in the separation $(A, B)$ is maximal.
To keep notation simple we will refer to our chosen values simply as $\rho$ and $E$.

Requirement \ref{Dclockwise} is satisfied by fiat because of the way we oriented $\Delta$.

 \item $V(W') \subseteq V(B)$ \\
 We showed that $W'$ contains a node of $\rho$. Choose a node $n$ such that $\pi(n) \in V(W')$.
 For $i = 1, \dots, 4$, let $R_i$ be a path in $W'$ connecting $y_i$ to $n$, and define $\tilde{P}_i = P_iR_i$.
 The path $\tilde{P_i}$ is grounded and intersects $D$ in a non-empty path, and therefore its track intersects the track of $D$ in a segment.
 The $D$-nodes in the intersection all appear on $\tilde{P}_i$ before it reaches $W'$ and therefore they are all in $P_i$.
 
 Let $d_i$ be the first node of $D$ on $P_i$. If we assume that the sequence $x_1, \dots, x_4$ is ordered in $C$-order, then the tracks $\tr(P_1[x_1, d_1])$ and $\tr(P_3[d_3, x_3])$ divide $\Delta_D^{\text{out}}$ into two regions.
 If $n$ is in one of these regions, then one of the tracks $\tr(\tilde{P}_2)$ and $\tr(\tilde{P}_4)$ is unable to reach $n$ without intersecting
 either $\tr(P_1[x_1, d_1])$, or  $\tr(P_3[d_3, x_3])$, or the interiors of both tracks $\tr(D[d_1, d_3])$ and $\tr(D[d_3, d_1])$. For convenience, assume that this track is $\tr(\tilde{P}_2)$.
 
 If the first case occurs then there is a node $n'$ in $\tr(P_1[x_1, d_1])$ that also belongs to $\tr(\tilde{P}_2)$. $\pi(n') \ne W'$ because $\pi(n')$ occurs on $P_1$ before it reaches $D$.
 Therefore it cannot occur in $R_2$ and therefore it must occur in $P_2$. But $P_1$ and $P_2$ do not intersect outside of $W'$, so this case is impossible. The second case is disposed of in exactly the same way.
 In the third case, $\tr(\tilde{P}_2)$ must intersect $\tr(D)$ in at least two disjoint segments.
 This implies that the intersection $P_2 \cap D$ is disconnected, which we assumed was not the case. 
 
 So we established that $n \in \diskin{D}$. Since $n \not\in V(D)$ by definition,
 it must be in the interior of $\diskin{D}$ and therefore it must belong to $\sigma(c)$ for some cell $c \subseteq \diskin{D} \subseteq \diskin{E}$. Therefore $n \in V(B)$.
The same holds for all the nodes of  $W'$. If $v$ is a non-node vertex of $W'$, then $v$ must belong to a unique cell $c$.
 Since W' is connected and not limited to a single cell, there must be a node of $W'$ in $c$.
 This node must be in $\diskin{E}$ and it cannot be on its boundary, because it is not a node of $D$ by definition, and therefore not a node of $E$ since $N(E) \subseteq N(D)$.
 Therefore it is in the interior of $\diskin{E}$ and therefore $c$ is an interior cell of $\diskin{E}$. it follows that $v \in V(B)$ and therefore $V(W') \subseteq V(B)$.
 
 \item $V(D) \subseteq V(B)$, and therefore $V(W) \subseteq V(B)$ \\
 By assumption, $\diskin{D} \subseteq \diskin{E}$. It follows immediately that all the nodes of $D$ are in $V(B)$.
 Let $v \in V(D)$ be a vertex that is not a node. Then $v$ is an interior vertex of some factor $Q_D$ of $D$.
 If $h(Q_D) \subseteq \diskin{E}$ then $v \in V(B)$. So we can assume that $h(Q_D)$ is exterior to $E$, and therefore exterior to $D$ as well. Since this cell contains a factor of $D$, it must be a border cell of $D$.
 The track of $Q_D$ separates the interior of $h(Q_D)$ (which is in $\Delta_E^{\text{out}}$ by assumption), from the interior of $\diskin{D}$ (because it is border cell and $\tr(Q_D)$ is part of $\tr(D)$ by definition).
 But $\diskin{D} \subseteq \diskin{E}$, so $\tr(Q_D)$ separates $\diskin{E}$ from $\Delta_E^{\text{out}}$. Therefore $\tr(Q_D)$ is part of $\tr(E)$, and its ends, which are consecutive nodes in $D$,
 must therefore be consecutive in $E$ as well. As a result, by property \ref{Efactors} of $E$, $Q_D$ is a factor of $E$ as well. Therefore $v \in (V(D) \cap V(E)) \subset V(B)$ and we are done.
 
  \item $\Omega \subseteq A \cap B \subseteq V(D)$ \\
  We start with the left inclusion. By definition the set $P$ is confined to the intersections $V(D) \cap \sigma(c)$ in cells $c$ that are border cells of both $D$ and $E$, where by assumption $D$ and $E$ coincide.
  Therefore $P \subseteq V(D) \cap V(E) \subseteq B$. Since any $E$-border cell $c$ is by definition a subset of $\Delta_E^{\text{out}}$, we also have $P \subseteq A$, and therefore $P \subseteq A \cap B$.
  In addition,
  \begin{align*}
  N(E) & \subseteq \tr(E) \subseteq \Delta_E^{\text{out}} \,\,\,\text{ and therefore }\,\,\, N(E) \subseteq N(\Delta_E^{\text{out}}) \subseteq A \\
  N(E) & \subseteq N(D) \cap V(E) \subseteq V(D) \cap V(E) \subseteq B
  \end{align*}
  And therefore $\Omega = N(E) \cup P \subseteq A \cap B$.
  
  To show the right inclusion, first observe that
  $$
  N(\Delta_E^{\text{out}}) \cap N(\diskin{E}) = N(\Delta_E^{\text{out}} \cap \diskin{E}) = N(\tr(E)) = N(E) \subseteq N(D) \subseteq V(D)
  $$
  and then break $A$ and $B$ into their constituents, and show the inclusion of the resulting intersections:
  \begin{alignat*}{2}
  N(\Delta_E^{\text{out}}) & \cap ( \bigcup\limits_{c \subseteq \diskin{E}} \sigma(c) ) & & \subseteq N(\Delta_E^{\text{out}}) \cap N(\diskin{E}) \subseteq V(D) \\
  ( \bigcup\limits_{c \subseteq \Delta_E^{\text{out}}} \sigma(c) ) & \cap ( \bigcup\limits_{c \subseteq \diskin{E}} \sigma(c) ) & & \subseteq N(\Delta_E^{\text{out}}) \cap N(\diskin{E}) \subseteq V(D) \\
  A & \cap (V(D) \cap V(E)) & & \subseteq V(D)
  \end{alignat*}
    
  \item $E$ is proper. \\
  If $E$ is not proper, then there is a border cell $c$ of $E$ with one of the two following properties:
  \begin{itemize}
  \item $\vert \tilde{c} \vert = 2$
  \item $\vert \tilde{c} \vert = 3$ and $\tilde{c} \subseteq N(E)$
  \end{itemize}
  If the first case occurs, let $\rho'$ be a modification of $\rho$ where the modified tie-breaker function $\tau'$ chooses the other segment of $\bd(c)$ as the preferred segment.
  With this change the pair $(\rho', E)$ is still an eligible pair in Step \ref{step:ChooseRhoE} but with a strictly larger graph $B$, contrary to the choice of $\rho$ and $E$.
  Therefore this case does not occur.
  
 In the second case, let the nodes of $c$ be $m, n, p$ listed in clockwise $E$-order as they appear on $\tr(E)$.
 Notice that since $c$ is external to $\diskin{E}$, the clockwise $\bd(c)$-order of the three nodes is the opposite order.
 As a border cell, $\sigma(c)$ contains one or two factors of $E$, so it contains at least one edge.
 None of the nodes of $c$ are isolated in $G$ since they all belong to $V(E)$.
 Therefore by Lemma \ref{TightRenditionProps}(\ref{PropThreeConnected}) the nodes of $c$ belong to a single connected component of $\sigma(c)$.

Apply Lemma \ref{lem:twononcrossingloops} to $\tr(E)$ and $\bd(c)$.
 
 If $\tr(E)$ contains two segments along $\bd(c)$, we may assume, by rotating the names of the nodes, that these are $\bd(c)[n, m]$ and $\bd(c)[p, n]$.
 Lemma \ref{lem:twononcrossingloops} guarantees that  the loop
 $$
 L_1 = \bd(c)[m, p] \tr(E)[p, m]
 $$
 has an interior disk that contains both $c$ and $\diskin{E}$.
 
 If $\tr(E)$ contains only one segment along $\bd(c)$, we can assume, by rotating node names, that this segment is $\bd(c)[n, m]$.
 Lemma \ref{lem:twononcrossingloops} guarantees that one of the two loops
 \begin{align*}
 L_0 & = \bd(c)[p, n] \tr(E)[n, p] \\
 L_1 & = \bd(c)[m, p] \tr(E)[p, m]
 \end{align*}
 has an interior disk that contains both $c$ and $\diskin{E}$. If we are lucky, the desired loop is $L_1$.
 If not, we can rotate the node names one more time, making $L_1$ the desired loop while the segment of $E$ along $\bd(c)$ becomes $\bd(c)[p, n]$.
 
 With these naming conventions, in all cases the segment(s) of $\tr(E)$ along $\bd(c)$ would be either $\bd(c)[n, m]$, $\bd(c)[p, n]$ or both.
 The respective factor(s) of $E$ in $c$ are $E[m, n]$, $E[n, p]$ or both.
 
 By Lemma \ref{TightRenditionProps}(\ref{PropTwoConnected}), there is a path $R$ in $\sigma(c)$ between the nodes $m$ and $p$ that avoids $n$.
 
 Let $E'$ be the modification of $E$ created by replacing the path $E[m, p]$ with $R$. This removes the potential factors $E[m, n]$ and $E[n, p]$ from $E'$ because $E[m, p] = E[m, n]E[n, p]$, due to the node ordering.
 As a result $E'$ has no self intersections and is therefore a simple cycle. By construction, $N(E') \subseteq N(E) \subseteq N(D)$.
 
 $E'$ is grounded  since $R$ has at least one edge, which is in $c$, and the path $E'[p, m] = E[p, m]$ has at last one edge, which is not in $c$.
 By construction, $\tr(E') = L_1$ and therefore the clockwise order on $E'$ agrees with the clockwise order on $E$, and therefore with the clockwise order on $D$.
 If $s \rightarrow t$ is a consecutive pair of nodes in $E'$ which is also a consecutive pair in $D$ then that pair is not $m \rightarrow p$ and therefore $E'[s, t] = E[s, t] = D[s, t]$.
 
 Finally, the by the property of $L_1$,  $\diskin{E} \subsetneq \diskin{L_1} = \diskin{E'}$, and $B$ becomes strictly larger by gaining the path $R$ without losing the node $n$.
 It follows that $E'$ meets all the criteria of the lemma, but it violates the maximality of $B$, a contradiction. Therefore $E$ must be proper.
 
 \item $(G[B], \Omega)$ is a rural society. \\
 We follow the same recipe as in Lemma 5.1 in \cite{NewProof}. We remove all the cells from $\rho$ that are neither interior to $\diskin{E}$ nor border cells of $E$.
 We remove all the nodes that do not belong to $\diskin{E}$ or abut a border cell.
 For each border cell $c$ of $E$, it follows from the propriety of $E$ that $\vert \tilde{c} \vert = 3 $ and exactly two of the nodes of $c$, $\alpha_c$ and $\beta_c$, are in $V(E)$. Denote its third node $z_c$.
 
 Redraw each border cell $c$ as in Figure \ref{fig:CarvingBorderCell} by first drawing a bisecting line $\ell^1_c$ through $c$ from $\alpha_c$ to $\beta_c$.
 This line carves $c$ into two regions, one of which is disjoint from $z_c$ and is denoted $C$ in the figure, with the other region denoted $B$.
Let the redrawn cell $\hat{c}$ be the region $C$.
 If $c$ is a border cell of $D$ and $\vert \sigma(c) \cap P \vert = \{ p \}$, draw a new node on the interior of $\ell^1_c$ and identify it with the vertex $p$.
 After modifying all the border cells, remove all the nodes $z_c$ from the drawing.
 
 Draw a line $\ell^2_c$ through region $B$ of $c$ from $\alpha_c$ to $\beta_c$ as in Figure  \ref{fig:CarvingBorderCell}.
 If the interior of $\ell^1_c$ has a new node $p$, make sure that $\ell^2_c$ passes through that node. Otherwise $\ell^2_c$ must be internally disjoint from $\ell^1_c$.
 By propriety, $\tr(E)$ has a single segment in each border cell $c$.
 Replace the segment of $\tr(E)$ in $c$ with $\ell^2_c$. The resulting curve is a simple loop. Let $\Delta'$ be the closed interior of this loop.
 Then $\Delta'$ contains all the interior cells and modified border cells of $E$, and the nodes on its boundary are exactly the points of $\Omega$, in the clockwise direction of $E$.
 
As a final step we redefine the flaps of each border cell $\hat{c}$ by defining $\sigma'(\hat{c}) = \sigma(c) \cap G[B]$. We leave the flaps of interior cells intact.
Notice that the propriety of $E$ implies that $z_c$ is not a vertex of $\sigma'(\hat{c})$.
 
 Altogether, this construction creates a rendition $\rho'$ on the disk $\Delta'$ of the rural society $(\bigcup\limits_{c \subseteq \Delta'} \sigma'(c), \,\,\Omega)$. We just have to show that
 $$
 \bigcup\limits_{c \subseteq \Delta'} \sigma'(c) = G[B]
 $$
 The inclusion $\bigcup\limits_{c \subseteq \Delta'} \sigma'(c) \subseteq G[B]$ is obvious.
 Conversely,
 $$
 \bigcup\limits_{c \subseteq \diskin{E}} \sigma(c) \subseteq \bigcup\limits_{c \subseteq \Delta'} \sigma'(c)
 $$
 and every vertex $v \in V(E)$ is either in $\sigma(c)$ for an internal cell $c$, or it is in $\sigma'(\hat{c})$ for a border cell $c$ of $\rho$, since we know that $v \ne z_c$.
 Therefore
 $$
 V(D) \cap V(E) \subseteq V(E) \subseteq \bigcup\limits_{c \subseteq \Delta'} \sigma'(c)
 $$
and so $B \subseteq \bigcup\limits_{c \subseteq \Delta'} \sigma'(c)$.
 
 To complete the proof we just need to verify that all the edges of $G[B]$ are accounted for.
 
Let $e \in G[B]$. Then there is a unique cell $c$ in $\rho$ such that $e \in \sigma(c)$. If $c \subset \diskin{E}$ then $e \in E(B)$ and we are done.
If $c$ is a border cell of $\tr(E)$ then $e \in \sigma'(\hat{c})$ by definition.
We are left with the case where $c$ is an external cell which is not a border. In this case there is no factor of $E$ with a home in $c$, and therefore the ends of $e$ must be in $\tr(E)$, which means that the ends of $e$ are nodes.
Therefore by Lemma \ref{TightRenditionProps}(\ref{PropHomeAlone}), $c$ must be trivial. Let $s$ and $t$ be the nodes of $c$.

Apply Lemma \ref{lem:twononcrossingloops} to $\tr(E)$ and $\bd(c)$.

 The lemma guarantees that one of the two loops
 \begin{align*}
 L_0 & = \bd(c)[s, t] \tr(E)[t, s] \\
 L_1 & = \bd(c)[t, s] \tr(E)[s, t]
 \end{align*}
 has an interior disk that contains both $c$ and $\diskin{E}$. By flipping the names $s$ and $t$ if necessary, we can guarantee that the desired loop is $L_1$. 

Create a cycle $E'$ by replacing the path $E[t, s]$ with the edge $e$.
Create a $C$-rendition $\rho'$ of $G$ by modifying the tie-breaking function $\tau$, if necessary, such that $\tau'(c)$ chooses the boundary segment of $\bd(c)$ that is used by $L_1$, namely $\bd(c)[t, s]$.

With this modification, we have $\tr_{\rho'}(E') = L_1$.

Using the same arguments we used in the proof of propriety of $E$ we can conclude that that $E'$ is a simple cycle which is $\rho'$-grounded,
with the same clockwise direction as the direction inherited from the clockwise direction of $D$, with $N(E') \subset N(D)$, and with the same factors as $D$ for nodes that are consecutive in both $D$ and $E'$.
As before, $\diskin{E} \subsetneq \diskin{L_1} = \diskin{E'}$, and $B$ grows by adding the edge $e$. It follows that the pair $(\rho', E')$ violates the maximality of $(\rho, E)$, and this case cannot occur. This concludes the proof.
\end{enumerate}
 \end{proof}

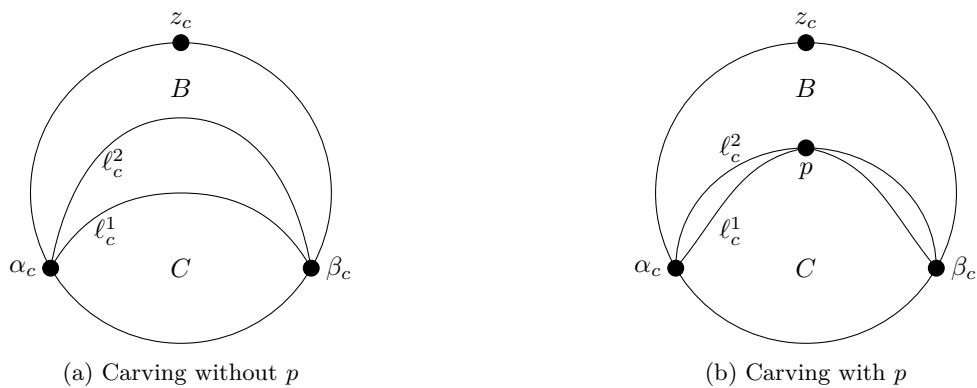
\begin{figure}[htb]
\centering
\begin{subfigure}{.45\linewidth}
    \centering
    \begin{tikzpicture}
    \draw (0, 0) circle (2);

    \filldraw (1.732, -1) circle (3 pt);
    \node at (2.1, -1) {$\beta_c$};
    \filldraw (-1.732, -1) circle (3 pt);
    \node at (-2.1, -1) {$\alpha_c$};
    \filldraw (0, 2) circle (3 pt);
    \node at (0, 2.3) {$z_c$};

    \node at (-1, -0.5) {$\ell^1_c$};
    \node at (-0.9, 0.4) {$\ell^2_c$};

    \node at (0, 1.4) {$B$};
    \node at (0, -1) {$C$};

    \draw (-1.732, -1) to[in = 180, out = 60] (0, 0);
    \draw (0, 0) to[in = 120, out = 0] (1.732, -1);

    \draw (-1.732, -1) to[in = 180, out = 80] (0, 1);
    \draw (0, 1) to[in = 100, out = 0] (1.732, -1);

    \end{tikzpicture}
    \caption{Carving without $p$}
    \label{fig:CarvingBorderCellWithoutP}
\end{subfigure}
\hfill
\begin{subfigure}{.45\linewidth}
    \centering
    \begin{tikzpicture}
    \draw (0, 0) circle (2);

    \filldraw (1.732, -1) circle (3 pt);
    \node at (2.1, -1) {$\beta_c$};
    \filldraw (-1.732, -1) circle (3 pt);
    \node at (-2.1, -1) {$\alpha_c$};
    \filldraw (0, 2) circle (3 pt);
    \node at (0, 2.3) {$z_c$};
    \filldraw (0, 0.6) circle (3 pt);
    \node at (0, 0.3) {$p$};

    \node at (-1, -0.5) {$\ell^1_c$};
    \node at (-1, 0.6) {$\ell^2_c$};

    \node at (0, 1.4) {$B$};
    \node at (0, -1) {$C$};

    \draw (-1.732, -1) to[in = 190, out = 50] (0, 0.6);
    \draw (0, 0.6) to[in = 130, out = -10] (1.732, -1);

    \draw (-1.732, -1) to[in = 180, out = 90] (0, 0.6);
    \draw (0, 0.6) to[in = 90, out = 0] (1.732, -1);

    \end{tikzpicture}
    \caption{Carving with $p$}
    \label{fig:CarvingBorderCellWithP}
\end{subfigure}
\caption{Carving a border cell $c$}
\label{fig:CarvingBorderCell}
\end{figure}

\subsection{Where the rain gets in: Theorems 5.2 and 6.1 of \cite{NewProof} revisited}

We are now ready to fix the two main results in \cite{NewProof}, the Flat Wall Theorem (5.2) and the hereditary property of flat walls (Theorem 6.1).
We need a technical lemma that does the bulk of the work for both.
The lemma relies on 5.1' to show that under mild assumptions, a wall $W$ in a rural society $(G, C)$ is flat (as in Definition \ref{def:flatwall}.)

The main challenge in the proof of the lemma is to show that under the right circumstances, any maximal choice of the set $P$ of peg choices in 5.1'
yields a circular order $\Omega$ that contains a peg choice from each peg interval of $W$.

\begin{figure}[htb]
\centering
\begin{subfigure}{.45\linewidth}
    \centering
  \scalebox{0.5} {
    \begin{tikzpicture}
    \draw (3, 0) -- (7, 0);
    \draw (0, 2) -- (10.5, 2);
    \draw (0, 4) -- (1, 4);
    \draw [line width = 5pt] (1, 4) -- (9, 4);
    \draw (9, 4) -- (10.5, 4);
    \draw (1, 4) -- (1, 2);
    \draw (9, 4) -- (9, 2);
    \draw [dashed] (5, 5.5) -- (5, 4);
    \draw (5, 2) -- (5, 0);

    \filldraw [white] (5,6) circle (3pt);

    \filldraw [black] (1,4) circle (3pt);
    \node at (1, 4.4) {\Large $\alpha$};
    \filldraw [black] (2,4) circle (6pt);
    \filldraw [black] (3,4) circle (6pt);
    \filldraw [black] (5,4) circle (6pt);
    \node at (5, 3.5) {\Large $m$};
    \filldraw [black] (7,4) circle (6pt);
    \filldraw [black] (9,4) circle (3pt);
    \node at (9, 4.4) {\Large $\beta$};

    \filldraw [black] (1,2) circle (3pt);
    \filldraw [black] (5,2) circle (3pt);
    \filldraw [black] (9,2) circle (3pt);

    \filldraw [black] (5,0) circle (3pt);

    \end{tikzpicture}
  }
    \caption{A Top or bottom brick}
\end{subfigure}
    \hfill
\begin{subfigure}{.45\linewidth}
    \centering
  \scalebox{0.5} {
  \begin{tikzpicture}
  \draw (3, 0) -- (7, 0);
  \draw (0, 2) -- (5, 2);
  \draw [line width = 5pt] (5, 2) -- (9, 2);
  \draw [dashed] (9, 2) -- (10.5, 2);
  \draw (0, 4) -- (5, 4);
  \draw [line width = 5pt] (5, 4) -- (9, 4);
  \draw [dashed] (9, 4) -- (10.5, 4);
  \draw (3, 6) -- (7, 6);

  \draw (5, 6) -- (5, 4);
  \draw (1, 4) -- (1, 2);
  \draw [line width = 5pt] (9, 4) -- (9, 2);
  \draw (5, 2) -- (5, 0);

  \filldraw [black] (5,6) circle (3pt);

  \filldraw [black] (1,4) circle (3pt);
  \filldraw [black] (5,4) circle (3pt);
  \node at (5, 3.6) {\Large $\alpha$};
  \filldraw [black] (7,4) circle (6pt);
  \filldraw [black] (9,4) circle (6pt);
  \node at (9, 4.5) {\Large $m$};

  \filldraw [black] (9,3) circle (6pt);

  \filldraw [black] (1,2) circle (3pt);
  \filldraw [black] (5,2) circle (3pt);
  \node at (5, 2.4) {\Large $\beta$};
  \filldraw [black] (6,2) circle (6pt);
  \filldraw [black] (8,2) circle (6pt);
  \filldraw [black] (5,2) circle (3pt);
  \filldraw [black] (9,2) circle (6pt);

  \filldraw [black] (5,0) circle (3pt);
  
  \end{tikzpicture}
  }
    \caption{A Left or right side brick}
\end{subfigure}

\bigskip
\begin{subfigure}{.45\linewidth}
    \centering
  \scalebox{0.5} {
  \begin{tikzpicture}
  \draw (3, 0) -- (7, 0);
  \draw (0, 2) -- (5, 2);
  \draw [line width = 5pt] (5, 2) -- (9, 2);
  \draw [dashed] (9, 2) -- (10.5, 2);
  \draw (0, 4) -- (1, 4);
  \draw [line width = 5pt] (1, 4) -- (9, 4);
  \draw [dashed] (9, 4) -- (10.5, 4);

  \draw (1, 4) -- (1, 2);
  \draw [line width = 5pt] (9, 4) -- (9, 2);
  \draw [dashed] (5, 5.5) -- (5, 4);
  \draw (5, 2) -- (5, 0);

  \filldraw [white] (5,6) circle (3pt);

  \filldraw [black] (1,4) circle (3pt);
  \node at (1, 4.4) {\Large $\alpha$};
  \filldraw [black] (2,4) circle (6pt);
  \filldraw [black] (3,4) circle (6pt);
  \filldraw [black] (5,4) circle (6pt);
  \filldraw [black] (7,4) circle (6pt);
  \filldraw [black] (9,4) circle (6pt);
  \node at (9, 4.5) {\Large $m$};

  \filldraw [black] (9,3) circle (6pt);

  \filldraw [black] (1,2) circle (3pt);
  \filldraw [black] (5,2) circle (3pt);
  \node at (5, 2.4) {\Large $\beta$};
  \filldraw [black] (6,2) circle (6pt);
  \filldraw [black] (8,2) circle (6pt);
  \filldraw [black] (5,2) circle (3pt);
  \filldraw [black] (9,2) circle (6pt);

  \filldraw [black] (5,0) circle (3pt);
  
  \end{tikzpicture}
  }
    \caption{A Bulging corner brick}
\end{subfigure}
\hfill
\begin{subfigure}{.45\linewidth}
    \centering
  \scalebox{0.5} {
  \begin{tikzpicture}
  \draw (3, 0) -- (7, 0);
  \draw (0, 2) -- (10.5, 2);
  \draw (0, 4) -- (1, 4);
  \draw [line width = 5pt] (1, 4) -- (9, 4);

  \draw (1, 4) -- (1, 2);
  \draw [line width = 5pt] (9, 4) -- (9, 2);
  \draw [dashed] (5, 5.5) -- (5, 4);
  \draw (5, 2) -- (5, 0);

  \filldraw [white] (5,6) circle (3pt);

  \filldraw [black] (1,4) circle (3pt);
  \node at (1, 4.4) {\Large $\alpha$};
  \filldraw [black] (2,4) circle (6pt);
  \filldraw [black] (3,4) circle (6pt);
  \filldraw [black] (5,4) circle (6pt);
  \node at (5, 3.5) {\Large $m$};
  \filldraw [black] (7,4) circle (6pt);
  \filldraw [black] (9,4) circle (6pt);

  \filldraw [black] (9,3) circle (6pt);

  \filldraw [black] (1,2) circle (3pt);
  \filldraw [black] (5,2) circle (3pt);
  \filldraw [black] (5,2) circle (3pt);
  \filldraw [black] (9,2) circle (3pt);
  \node at (9, 1.5) {\Large $\beta$};

  \filldraw [black] (5,0) circle (3pt);
  
  \end{tikzpicture}
  }
    \caption{A Recessed corner brick}
\end{subfigure}
\caption{Border bricks with peg intervals}
\label{fig:bricksandpegintervals}
\end{figure}
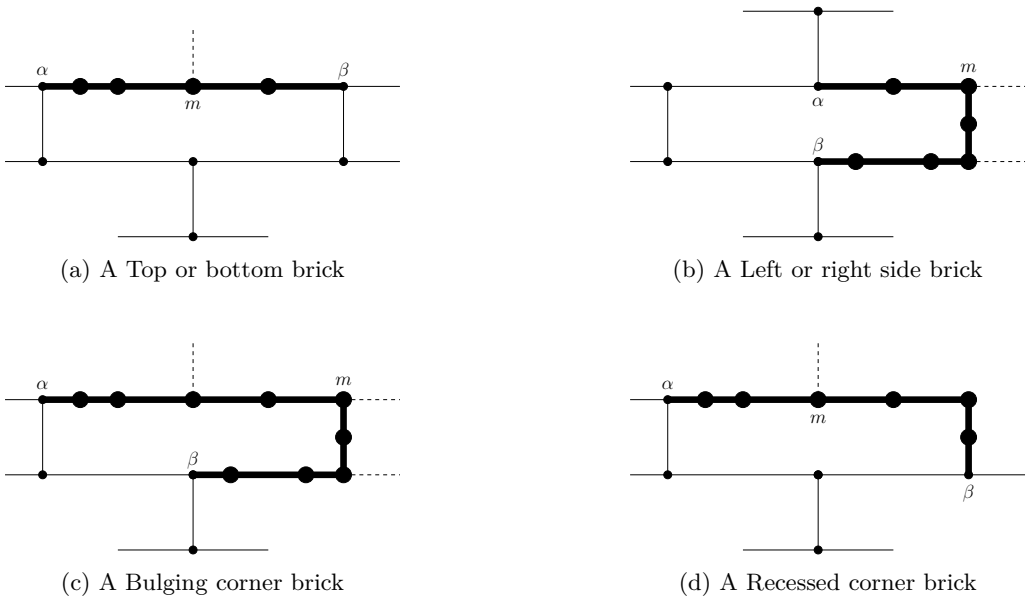


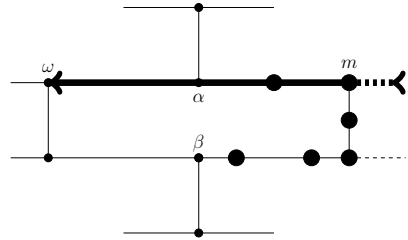
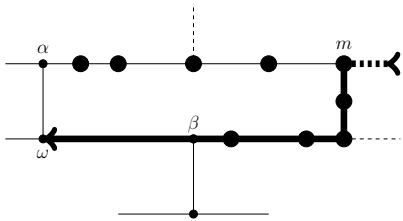
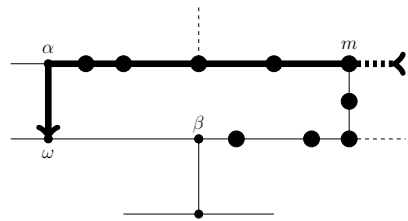
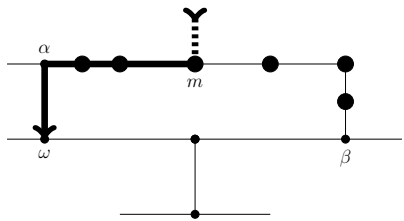
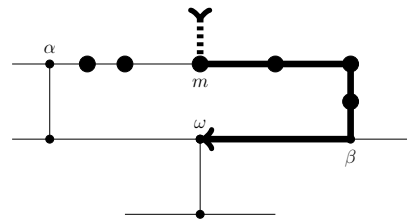
\begin{figure}[!htb]
\centering
\begin{subfigure}{.45\linewidth}
    \centering
  \scalebox{0.5} {
    \begin{tikzpicture}
    \draw (3, 0) -- (7, 0);
    \draw (0, 2) -- (10.5, 2);
    \draw (0, 4) -- (5, 4);
    \draw [line width = 5pt] (5, 4) -- (9, 4);
    \draw (9, 4) -- (10.5, 4);
    \draw (1, 4) -- (1, 2);
    \draw [-to, line width = 5pt] (9, 4) -- (9, 2);
    \draw [to reversed-, dashed, line width = 5pt] (5, 5.5) -- (5, 4);
    \draw (5, 2) -- (5, 0);

    \filldraw [white] (5,6) circle (3pt);

    \filldraw [black] (1,4) circle (3pt);
    \node at (1, 4.4) {\Large $\alpha$};
    \filldraw [black] (2,4) circle (6pt);
    \filldraw [black] (3,4) circle (6pt);
    \filldraw [black] (5,4) circle (6pt);
    \node at (5, 3.5) {\Large $m$};
    \filldraw [black] (7,4) circle (6pt);
    \filldraw [black] (9,4) circle (3pt);
    \node at (9, 4.4) {\Large $\beta$};

    \filldraw [black] (1,2) circle (3pt);
    \filldraw [black] (5,2) circle (3pt);
    \filldraw [black] (9,2) circle (3pt);
    \node at (9, 1.6) {\Large $\omega$};

    \filldraw [black] (5,0) circle (3pt);

    \end{tikzpicture}
  }
    \caption{Top or bottom brick, first pegging path}
\end{subfigure}
    \hfill
\begin{subfigure}{.45\linewidth}
    \centering
  \scalebox{0.5} {
    \begin{tikzpicture}
    \draw (3, 0) -- (7, 0);
    \draw (0, 2) -- (10.5, 2);
    \draw (0, 4) -- (1, 4);
    \draw [line width = 5pt] (1, 4) -- (5, 4);
    \draw (5, 4) -- (10.5, 4);
    \draw [-to, line width = 5pt] (1, 4) -- (1, 2);
    \draw (9, 4) -- (9, 2);
    \draw [to reversed-, dashed, line width = 5pt] (5, 5.5) -- (5, 4);
    \draw (5, 2) -- (5, 0);

    \filldraw [white] (5,6) circle (3pt);

    \filldraw [black] (1,4) circle (3pt);
    \node at (1, 4.4) {\Large $\alpha$};
    \filldraw [black] (2,4) circle (6pt);
    \filldraw [black] (3,4) circle (6pt);
    \filldraw [black] (5,4) circle (6pt);
    \node at (5, 3.5) {\Large $m$};
    \filldraw [black] (7,4) circle (6pt);
    \filldraw [black] (9,4) circle (3pt);
    \node at (9, 4.4) {\Large $\beta$};

    \filldraw [black] (1,2) circle (3pt);
    \node at (1, 1.6) {\Large $\omega$};
    \filldraw [black] (5,2) circle (3pt);
    \filldraw [black] (9,2) circle (3pt);

    \filldraw [black] (5,0) circle (3pt);

    \end{tikzpicture}
  }
    \caption{Top or bottom brick, second pegging path}
\end{subfigure}

\bigskip
\begin{subfigure}{.45\linewidth}
    \centering
  \scalebox{0.5} {
  \begin{tikzpicture}
  \draw (3, 0) -- (7, 0);
  \draw (0, 2) -- (1, 2);
  \draw [to-, line width = 5pt] (1, 2) -- (9, 2);
  \draw [dashed] (9, 2) -- (10.5, 2);
  \draw (0, 4) -- (9, 4);
  \draw [-to reversed, dashed, line width = 5pt] (9, 4) -- (10.5, 4);
  \draw (3, 6) -- (7, 6);

  \draw (5, 6) -- (5, 4);
  \draw (1, 4) -- (1, 2);
  \draw [line width = 5pt] (9, 4) -- (9, 2);
  \draw (5, 2) -- (5, 0);

  \filldraw [black] (5,6) circle (3pt);

  \filldraw [black] (1,4) circle (3pt);
  \filldraw [black] (5,4) circle (3pt);
  \node at (5, 3.6) {\Large $\alpha$};
  \filldraw [black] (7,4) circle (6pt);
  \filldraw [black] (9,4) circle (6pt);
  \node at (9, 4.5) {\Large $m$};

  \filldraw [black] (9,3) circle (6pt);

  \filldraw [black] (1,2) circle (3pt);
  \node at (1, 1.6) {\Large $\omega$};
  \filldraw [black] (5,2) circle (3pt);
  \node at (5, 2.4) {\Large $\beta$};
  \filldraw [black] (6,2) circle (6pt);
  \filldraw [black] (8,2) circle (6pt);
  \filldraw [black] (5,2) circle (3pt);
  \filldraw [black] (9,2) circle (6pt);

  \filldraw [black] (5,0) circle (3pt);
  
  \end{tikzpicture}
  }
    \caption{Left or right side brick, first pegging path}
\end{subfigure}
\hfill
\begin{subfigure}{.45\linewidth}
    \centering
  \scalebox{0.5} {
  \begin{tikzpicture}
  \draw (3, 0) -- (7, 0);
  \draw (0, 2) -- (9, 2);
  \draw [dashed] (9, 2) -- (10.5, 2);
  \draw (0, 4) -- (1, 4);
  \draw [to-, line width = 5pt] (1, 4) -- (9, 4);
  \draw [-to reversed, dashed, line width = 5pt] (9, 4) -- (10.5, 4);
  \draw (3, 6) -- (7, 6);

  \draw (5, 6) -- (5, 4);
  \draw (1, 4) -- (1, 2);
  \draw (9, 4) -- (9, 2);
  \draw (5, 2) -- (5, 0);

  \filldraw [black] (5,6) circle (3pt);

  \filldraw [black] (1,4) circle (3pt);
  \node at (1, 4.4) {\Large $\omega$};
  \filldraw [black] (5,4) circle (3pt);
  \node at (5, 3.6) {\Large $\alpha$};
  \filldraw [black] (7,4) circle (6pt);
  \filldraw [black] (9,4) circle (6pt);
  \node at (9, 4.5) {\Large $m$};

  \filldraw [black] (9,3) circle (6pt);

  \filldraw [black] (1,2) circle (3pt);
  \filldraw [black] (5,2) circle (3pt);
  \node at (5, 2.4) {\Large $\beta$};
  \filldraw [black] (6,2) circle (6pt);
  \filldraw [black] (8,2) circle (6pt);
  \filldraw [black] (5,2) circle (3pt);
  \filldraw [black] (9,2) circle (6pt);

  \filldraw [black] (5,0) circle (3pt);
  
  \end{tikzpicture}
  }
    \caption{Left or right side brick, second pegging path}
\end{subfigure}

\bigskip
\begin{subfigure}{.45\linewidth}
    \centering
  \scalebox{0.5} {
  \begin{tikzpicture}
  \draw (3, 0) -- (7, 0);
  \draw (0, 2) -- (1, 2);
  \draw [to-, line width = 5pt](1, 2) -- (9, 2);
  \draw [dashed] (9, 2) -- (10.5, 2);
  \draw (0, 4) -- (9, 4);
  \draw [-to reversed, dashed, line width=5pt] (9, 4) -- (10.5, 4);

  \draw (1, 4) -- (1, 2);
  \draw [line width = 5pt] (9, 4) -- (9, 2);
  \draw [dashed] (5, 5.5) -- (5, 4);
  \draw (5, 2) -- (5, 0);

  \filldraw [white] (5,6) circle (3pt);

  \filldraw [black] (1,4) circle (3pt);
  \node at (1, 4.4) {\Large $\alpha$};
  \filldraw [black] (2,4) circle (6pt);
  \filldraw [black] (3,4) circle (6pt);
  \filldraw [black] (5,4) circle (6pt);
  \filldraw [black] (7,4) circle (6pt);
  \filldraw [black] (9,4) circle (6pt);
  \node at (9, 4.5) {\Large $m$};

  \filldraw [black] (9,3) circle (6pt);

  \filldraw [black] (1,2) circle (3pt);
  \node at (1, 1.6) {\Large $\omega$};
  \filldraw [black] (5,2) circle (3pt);
  \node at (5, 2.4) {\Large $\beta$};
  \filldraw [black] (6,2) circle (6pt);
  \filldraw [black] (8,2) circle (6pt);
  \filldraw [black] (5,2) circle (3pt);
  \filldraw [black] (9,2) circle (6pt);

  \filldraw [black] (5,0) circle (3pt);
  
  \end{tikzpicture}
  }
    \caption{Bulging corner brick, first pegging path}
\end{subfigure}
\hfill
\begin{subfigure}{.45\linewidth}
    \centering
  \scalebox{0.5} {
  \begin{tikzpicture}
  \draw (3, 0) -- (7, 0);
  \draw (0, 2) -- (9, 2);
  \draw [dashed] (9, 2) -- (10.5, 2);
  \draw (0, 4) -- (1, 4);
  \draw [line width = 5pt] (1, 4) -- (9, 4);
  \draw [-to reversed, dashed, line width=5pt] (9, 4) -- (10.5, 4);

  \draw [-to, line width=5pt] (1, 4) -- (1, 2);
  \draw (9, 4) -- (9, 2);
  \draw [dashed] (5, 5.5) -- (5, 4);
  \draw (5, 2) -- (5, 0);

  \filldraw [white] (5,6) circle (3pt);

  \filldraw [black] (1,4) circle (3pt);
  \node at (1, 4.4) {\Large $\alpha$};
  \filldraw [black] (2,4) circle (6pt);
  \filldraw [black] (3,4) circle (6pt);
  \filldraw [black] (5,4) circle (6pt);
  \filldraw [black] (7,4) circle (6pt);
  \filldraw [black] (9,4) circle (6pt);
  \node at (9, 4.5) {\Large $m$};

  \filldraw [black] (9,3) circle (6pt);

  \filldraw [black] (1,2) circle (3pt);
  \node at (1, 1.6) {\Large $\omega$};
  \filldraw [black] (5,2) circle (3pt);
  \node at (5, 2.4) {\Large $\beta$};
  \filldraw [black] (6,2) circle (6pt);
  \filldraw [black] (8,2) circle (6pt);
  \filldraw [black] (5,2) circle (3pt);
  \filldraw [black] (9,2) circle (6pt);

  \filldraw [black] (5,0) circle (3pt);
  
  \end{tikzpicture}
  }
    \caption{Bulging corner brick, second pegging path}
\end{subfigure}

\bigskip
\begin{subfigure}{.45\linewidth}
    \centering
  \scalebox{0.5} {
  \begin{tikzpicture}
  \draw (3, 0) -- (7, 0);
  \draw (0, 2) -- (10.5, 2);
  \draw (0, 4) -- (1, 4);
  \draw [line width = 5pt] (1, 4) -- (5, 4);
  \draw (5, 4) -- (9, 4);

  \draw [-to, line width = 5pt] (1, 4) -- (1, 2);
  \draw (9, 4) -- (9, 2);
  \draw [to reversed-, dashed, line width = 5pt] (5, 5.5) -- (5, 4);
  \draw (5, 2) -- (5, 0);

  \filldraw [white] (5,6) circle (3pt);

  \filldraw [black] (1,4) circle (3pt);
  \node at (1, 4.4) {\Large $\alpha$};
  \filldraw [black] (2,4) circle (6pt);
  \filldraw [black] (3,4) circle (6pt);
  \filldraw [black] (5,4) circle (6pt);
  \node at (5, 3.5) {\Large $m$};
  \filldraw [black] (7,4) circle (6pt);
  \filldraw [black] (9,4) circle (6pt);

  \filldraw [black] (9,3) circle (6pt);

  \filldraw [black] (1,2) circle (3pt);
  \node at (1, 1.6) {\Large $\omega$};
  \filldraw [black] (5,2) circle (3pt);
  \filldraw [black] (5,2) circle (3pt);
  \filldraw [black] (9,2) circle (3pt);
  \node at (9, 1.5) {\Large $\beta$};

  \filldraw [black] (5,0) circle (3pt);
  
  \end{tikzpicture}
  }
    \caption{Recessed corner brick, first pegging path}
\end{subfigure}
\hfill
\begin{subfigure}{.45\linewidth}
    \centering
  \scalebox{0.5} {
  \begin{tikzpicture}
  \draw (3, 0) -- (7, 0);
  \draw (0, 2) -- (5, 2);
  \draw [to-, line width = 5pt] (5, 2) -- (9, 2);
  \draw (9, 2) -- (10.5, 2);
  \draw (0, 4) -- (5, 4);
  \draw [line width = 5pt] (5, 4) -- (9, 4);

  \draw (1, 4) -- (1, 2);
  \draw [line width = 5pt] (9, 4) -- (9, 2);
  \draw [to reversed-, dashed, line width = 5pt] (5, 5.5) -- (5, 4);
  \draw (5, 2) -- (5, 0);

  \filldraw [white] (5,6) circle (3pt);

  \filldraw [black] (1,4) circle (3pt);
  \node at (1, 4.4) {\Large $\alpha$};
  \filldraw [black] (2,4) circle (6pt);
  \filldraw [black] (3,4) circle (6pt);
  \filldraw [black] (5,4) circle (6pt);
  \node at (5, 3.5) {\Large $m$};
  \filldraw [black] (7,4) circle (6pt);
  \filldraw [black] (9,4) circle (6pt);

  \filldraw [black] (9,3) circle (6pt);

  \filldraw [black] (1,2) circle (3pt);
  \filldraw [black] (5,2) circle (3pt);
  \node at (5, 2.4) {\Large $\omega$};
  \filldraw [black] (5,2) circle (3pt);
  \filldraw [black] (9,2) circle (3pt);
  \node at (9, 1.5) {\Large $\beta$};

  \filldraw [black] (5,0) circle (3pt);
  
  \end{tikzpicture}
  }
  \caption{Recessed corner brick, second pegging path}
\end{subfigure}
\caption{Border bricks with their pegging paths}
\label{fig:bricksandpeggingpaths}
\end{figure}

\begin{lem}\label{lem:Wiswallin5.1}
Let $(G, C)$ be a rural society, $W \subseteq G$ a wall of height $r \ge 3$, and $D \subset W$ the boundary of $W$.
Assume that each peg interval $I$ of $D$ has a simple path $R_I$ from $C$ to an interior vertex of $I$, such that $R_I$ does not intersect $V(W)$ except at its $I$ terminus.
Let $I_1, \dots, I_4$ be the peg intervals of the corner bricks of $W$. Assume that $R_{I_1}, \dots, R_{I_4}$ are vertex disjoint. Then $W$ is flat.
\end{lem}

 \begin{proof}
The peg intervals of $W$ occur along $D$. Figure \ref{fig:bricksandpegintervals} shows all the possible types of $W$ border bricks that carry a peg interval along their boundaries.
The dashed lines are possible edges of $G$ outside of $W$ that connect to its peg intervals, while $W$-edges are shown as solid lines.
The peg intervals themselves are highlighted, and their ends are marked by $\alpha$ and $\beta$, so that each depicted peg interval, as shown, is $D[\alpha, \beta]$ in the clockwise direction.
In the interior of each peg interval $D[\alpha, \beta]$ the terminus of $R_I$ is marked as $m = m_{\alpha \beta}$. While $m$ has degree $3$ in $G$ it only has degree $2$ in $W$.

The peg intervals of reflected bricks (e.g. bottom bricks and right side bricks) are $D[\beta, \alpha$] in the clockwise direction. The following analysis applies to the bricks as depicted.
To analyze the reflected bricks, $\alpha$ and $\beta$ need to be interchanged.
Notice that the boundary $D$ also passes along recessed side bricks, that are not depicted. While these are border bricks, they do not possess peg intervals, since in the elementary wall they do not have degree $2$ vertices.

We start by constructing {\em pegging paths} $S^{\alpha}_I$ and $S^{\beta}_I$ for each peg interval (see Figure \ref{fig:bricksandpeggingpaths} for pegging paths.)
For each border brick $B$ of $W$ with peg interval $I$, we construct $S^{\alpha}_I$ and $S^{\beta}_I$ along the boundary of $B$,
both starting at the terminus $m$ of $R_I$, continuing towards the vertex $\alpha$ or $\beta$, respectively, and ending at the vertex $\omega$.
Notice that $\omega \in V(W) \setminus V(D)$.

Apply Lemma 5.1' to $(G, C)$, $W$, $D$ and the concatenated paths $P_1 = R_{I_1}S^{\alpha}_{I_1}, \dots, P_4 = R_{I_4}S^{\alpha}_{I_4}$.
It is not hard to see that all the conditions of the lemma are met. In particular $\vert C \vert \ge 4$ because $R_{I_1}, \dots, R_{I_4}$ are vertex disjoint.
We can conclude that there is a rendition $\rho$ that makes $D$ grounded, and a proper grounded cycle $E$ that meet the conclusions of Lemma 5.1'.
Using the notation of 5.1', the separation $(A, B)$ and the circular order $\Omega$ almost prove that $W$ is flat.
The only thing left to prove is that the set $P$ of peg choices can be chosen such that $\Omega$ contains a peg choice from each peg interval.
We will show that this condition holds for all maximal choices of $P$.

The proof of Lemma 5.1' establishes that there is a node $n$ of $W \setminus V(D)$ in the interior of $\diskin{D}$.
For each peg interval define $T_I$ to be a path in $W \setminus V(D)$ leading from $\omega$ to $n$ . Define
\begin{align*}
P^{\alpha}_I  & = R_I S^{\alpha}_I T_I  \\
P^{\beta}_I  & = R_I S^{\beta}_I T_I 
\end{align*}

It is easy to check that
\begin{align*}
I & = (P^{\alpha}_I \cup P^{\beta}_I) \cap D \\ 
\alpha & \in V(P_{\alpha}) \setminus V(P_{\beta}) \\
\beta & \in V(P_{\beta}) \setminus V(P_{\alpha})
\end{align*}

We need to show that $\Omega$ contains a peg choice for every peg interval in $D$. In other words, $\Omega$ must intersect the interior of each peg interval.
Suppose that is not the case, and let $I = D[\alpha, \beta]$ be a peg interval whose interior does not intersect $\Omega$. Without loss of generality we assume that this interval is depicted in Figure \ref{fig:bricksandpegintervals}.
\begin{enumerate}[Step 1:]

\item $\alpha$ and $\beta$ are consecutive nodes of $E$. \\
Since $P^{\alpha}_I$ is grounded and leads from $\Delta_E^{\text{out}}$ to $\diskin{E}$, there must be a node in the intersection of $\tr(E)$ and $\tr(P^{\alpha}_I)$.
By our assumption, this node cannot be internal to $I$, and since $N(E) \cap N(P^{\alpha}_I) \subseteq N(I) \setminus \{ \beta \}$, it follows that the node must be $\alpha$.
We repeat the same argument with $P^{\beta}_I$ to conclude that both $\alpha$ and $\beta$ are nodes of $E$.

If there is another node $\zeta$ of $E$ between $\alpha$ and $\beta$, then $\zeta$ is a node of $D$ (since $N(E) \subseteq N(D)$) and since the orders on $N(E)$ induced by the clockwise directions of $E$ and $D$ are identical,
$\zeta$ is between $\alpha$ and $\beta$ in $D$ as well. In other words, $\zeta \in N(E)$ is an internal node of the peg interval $I$, contrary to our assumption.
As consecutive nodes, $\alpha$ and $\beta$ are the ends of an $E$-factor $E[\alpha, \beta]$ that resides in some cell $c = h(E[\alpha, \beta])$.

\item If the cell $c$ has 3 nodes $\{ \alpha, \beta, \gamma \}$ then $\gamma \not \in V(I)$. \\
Suppose that $\gamma \in V(I)$.
As a node of $D$, $\gamma \in \tr(D)$ and since $\diskin{D} \subseteq \diskin{E}$ it follows that $\gamma \in \diskin{E}$.
Recall that $R_I$ is a path that connects $C$ to $m_{\alpha \beta}$.
Let $J$ be the sub-interval of $I$ connecting $m_{\alpha \beta}$ to $\gamma$ (so $J$ is either $D[m_{\alpha \beta}, \gamma]$ or $D[\gamma, m_{\alpha \beta}]$, depending on the order of $m_{\alpha \beta}$ and $\gamma$ in $D$).
The concatenation $R' = R_I \cdot J$ is a grounded path leading from $\diskout{E}$ to $\diskin{E}$.
As such, the track of $R'$ must intersect the track of $E$.
But neither $R_I$ nor $J$ contain a node of $E$. Recall that $V(R_I) \cap N(E)) \subseteq V(R_I) \cap V(W) = \{ m_{\alpha \beta} \} \subseteq V(J)$,
and $J$ is wholly contained in the interior of $I$ which contains no $E$ nodes, by assumption.

\item $I \cap \sigma(c) = \{ \alpha, \beta \}$ \\
We showed that $I$ does not contain nodes of $c$ except at its ends.
Therefore the interior of $I$ is either entirely inside $\sigma(c)$ or entirely outside.
Assume that $I \subseteq \sigma(c)$.
Then $I$ is a factor of $D$ and $\alpha$ and $\beta$ are consecutive nodes of $D$, and by our assumptions on $E$ we know that $I = D[\alpha, \beta] = E[\alpha, \beta]$.
Suppose that $c$ is a border cell of $E$. Since $\sigma(c)$ contains at least one internal vertex (because $m_{\alpha \beta} \in I$)
it follows from the maximality of $P$ that there is a vertex $p \in P \cap \sigma(c)$.
Since $E$ is proper, the third node $\gamma$ of $c$ is in the interior of $\diskout{E}$ and therefore not in $V(D)$ and therefore $D$ has a single factor in $c$.
It follows that $p \in P \cap V(I)$  contrary to our assumption that $\Omega$ does not contain a peg choice for $I$.
The cell $c$ cannot be exterior either because it contains a factor of $E$. Therefore $c \subset \diskin{E}$.

The path $R_I$ goes from $C$ in the exterior of $c$ to the vertex $m_{\alpha \beta} \in V(I) \subset V(\sigma(c))$.
Therefore $V(R_I)$ must contain a node $\gamma$ of $c$.
Since $c$ is interior to $E$, the grounded subpath of $R_I$ leading from $C$ to $\gamma$ must pass through a node of $E$.
Since $R_I$ is disjoint from $D$ except at $m_{\alpha \beta}$, it follows that this $E$ node must be $m_{\alpha \beta}$, contrary to our assumption that there are no $E$ nodes in the interior of $I$.

We conclude that the interior of $I$ is entirely outside $\sigma(c)$.

\item If $Q_D$ is a factor of $I$, then $h(Q_D) \subset \diskin{E}$. \\
Assume that $h(Q_D) \subset \diskout{E}$. Then it must be exterior to $D$  as well, and as the home of a factor of $D$ it must be a border cell of $D$, with $\tr(Q_D)$ separating the interior of $h(Q_D)$ from $\diskin{D}$.
Therefore $\tr(Q_D)$ separates $\Delta_E^{\text{out}}$ from $\diskin{E}$, and so $\tr(Q_D)$ is a segment of $\tr(E)$, and so there is a factor $Q_E$ of $E$ with $h(Q_E) = h(Q_D)$ and with $Q_D$ and $Q_E$ sharing the same ends.
Since the clockwise directions on $N(D)$ and $N(E)$ agree, $Q_E$ must be a subpath of $E[\alpha, \beta]$,
which implies that $Q_E = E[\alpha, \beta]$ and therefore $Q_D = D[\alpha, \beta] = I$ and so $h(I) = h(Q_D) = h(Q_E) = h(E[\alpha, \beta]) = c$,
contrary to our assumption that $I$ is disjoint from the interior of $c$.

\item Reach a contradiction and conclude that there is a peg choice for $I$ in $\Omega$. \\
Let $Q_D$ be a factor of $D[\alpha, \beta]$ that contains $m_{\alpha, \beta}$ as a vertex ($Q_D$ may not be unique because we cannot exclude the possibility that $m_{\alpha \beta}$ is a node.)
The path $R_I$ must intersect a node $n$ of $h(Q_D)$.
Since $h(Q_D) \subseteq \diskin{E}$, we have $n \in \diskin{E}$ and therefore the grounded subpath of $R_I$ from $C$ to $n$ must contain a node of $E$.
The only possible candidate for such a node is $m_{\alpha \beta}$, contrary to our assumption that there are no nodes of $E$ in the interior of $I$.
\end{enumerate}
\end{proof}

\subsubsection{Revisiting the proof of 5.2 (The Flat Wall Theorem)}

In the final part of the proof of 5.2, one obtains a rural society $(H_i, C)$ where $H_i$ contains a wall $W_i$ that contains the vertices of $C$ as corners, and a subwall $W \subset W_i$, with boundary $D$,
such that $W$ is far from the boundary of $W_i$.
The original proof of 5.2 is then proceeds by appealing to lemma 5.1, proving that $W$ is flat in $H_i$ (and ultimately in $G$), using the notion of flatness defined in $\cite{NewProof}$ .

To fix the proof, we use Lemma \ref{lem:Wiswallin5.1} instead. All we have to do is construct the paths $R_I$.
This is not hard to do, but requires slightly different constructions for peg intervals $I$ belonging to different types of border bricks $B$ of $W$.
The construction works because $W$ is contained entirely within the interior of $W_i$ and does not intersect its boundary.

If $B$ is a bulging right side brick or a top right corner brick of $W$ (either bulging or recessed), construct $R_{I}$
by first drawing a horizontal rightward path from the top right corner of $B$ (which is in the interior of $I$) to the first vertex $v$ on the boundary of $W_i$,
and then continue up the right boundary of $W_i$ to the top right corner of $W_i$ which is in $C$ by assumption.

The same construction holds for bulging left side bricks and bottom left corner brick of $W$ by rotating the picture 180 degrees,
it works for the bottom right corner brick by flipping the picture 180 degrees along the horizontal axis,
and for the top left corner brick by flipping the picture 180 degrees along the vertical axis.

The remaining case is when $B$ is a top (or bottom) brick of $W$.
Start $R_{I}$ at the upward $W_i$-edge emanating from the middle of $I$, and continue in a vertical, right-bulging square-wave pattern until you hit the boundary of $W_i$ at a vertex $v$,
and then continue right along the boundary of $W_i$ until you reach a corner, which is in $C$ by assumption.

It is not hard to check that these constructions give the desired paths and that the four corner brick paths are mutually disjoint as required.
Lemma \ref{lem:Wiswallin5.1} now implies that $W$ is flat in $H_i$.

\subsubsection{Revisiting Lemma 6.1 (Subwalls of flat walls are flat)}

Lemma 6.1 in \cite{NewProof} attempts to prove that a subwall of a flat wall $W$ is also flat, at least when the boundary of the subwall is disjoint from the boundary of $W$.
According to \cite{MoreAccurate}, this assertion is not true in full generality, and that paper proposes a way to add {\em certificates of flatness} to make the statement of 6.1 true after some necessary modifications.
With the new definition of flatness proposed here, we show that 6.1 is true in general, without any restrictions on the boundary of the subwall.

\begin{lem}
Let $W$ be a flat wall in a graph $G$, and let $W'$ be a subwall of $W$ of height at least 3. Then $W'$ is flat in $G$. 
\end{lem}

\begin{proof}
Let $D$ and $D'$ be the boundaries of $W$ and $W'$, respectively.
$W$ being flat means that there is a separation $(A, B)$ of $G$ and a vertex set $\Omega \subseteq A \cap B$ such that
\begin{itemize}
\item $V(W) \subseteq V(B)$
\item $A \cap B \subseteq V(D)$
\item $\Omega$ contains an internal vertex from each peg interval of $W$.
\item Endow $\Omega$ with a circular order induced from $D$. Then $(G[B], \Omega)$ is a rural society.
\end{itemize}

Apply lemma 5.1' to the society $(G[B], \Omega)$, the subgraph $W'$ and its cycle $D'$. Finding the required paths $P_1, \dots, P_4$ is trivial.
Then there is a $\Omega$-rendition $\rho$ of $G[B]$ and a $\rho$-proper cycle $E$ in $G[B]$ such that
\begin{itemize}
\item $N(E) \subseteq N(D')$
\item $\diskin{E} \supseteq \diskin{D'}$
\item The sets
\begin{itemize}
\item $A' = N(\diskout{E}) \cup \bigcup_{c \subseteq \Delta_E^{\text{out}}} \sigma(c)$
\item $B' = (V(D') \cap V(E)) \cup \bigcup_{c \subseteq \diskin{E}} \sigma(c)$
\end{itemize}
 form a separation of $G[B]$ such that $V(W') \subseteq V(B')$ and for any choice $P$ of internal vertices of shared factors of $D'$ and $E$,
 $N(E) \cup P \subseteq A' \cap B' \subseteq V(D')$ and $((G[B])[B'], N(E) \cup P)$ is rural.
 \end{itemize}

It is easy to see that $(G[B])[B'] = G[B']$ and so $(G[B'], N(E) \cup P)$ is rural.
Look at the separation $(\bar{A}, \bar{B}) = (A \cup A', B \cap B')$. It is obvious that $(\bar{A}, \bar{B})$ is a separation of $G$. We will show that together with $\Omega' = N(E) \cup P$ it provides evidence for the flatness of $W'$.

\begin{enumerate}[Step 1:]
\item $V(W') \subseteq V(\bar{B})$ \\
This is easy since we already know that $V(W') \subseteq V(W) \subseteq V(B)$ and $V(W') \subseteq V(B')$.
\item Choose $P$ to be maximal. Then $\Omega'$ intersects the interior of each peg interval of $W'$. \\

This will follow directly from Lemma \ref{lem:Wiswallin5.1}, once we show that all the required paths $R_I$ from $\Omega$ to the peg intervals of $W'$ exist.
This is not hard to do. Here is one recipe:

Let $I'$ be a peg interval of $W'$ along a border brick $B'$ of $W'$.
If $m \in V(I') \cap \Omega$ is an interior vertex of $I'$ then define $R_{I'} = \{ m \}$. Otherwise assume that $\Omega$ does not intersect the interior of $I'$.
This also implies that $B'$ is not a border brick of $W$.

If $B'$ is a bulging right side brick or a top right corner brick of $W'$, construct $R_{I'}$
by drawing a horizontal rightward path from the top right corner of $B'$ (which is in the interior of $I'$) to the first vertex $v$ on the boundary of $W$.
The vertex $v$ is a member of a unique $W$-brick $B$ which is either a bulging right side brick or a bulging top right corner brick of $W$, and is an end vertex of its $W$-peg interval $I$.
By assumption $\Omega$ contains an interior point $p$ of $I$. Continue the path from $v$ to $p$ along $I$, thus completing $R_{I'}$.
The same construction holds when "right" is replaced by "left" or "top" is replaced by "bottom".

The remaining case is when $B'$ is a top (or bottom) brick of $W'$.
By our assumption on $V(I') \cap \Omega$ we can assume that $B'$ is  an interior brick of $W$.
Start $R_{I'}$ at the upward $W$-edge emanating from the middle of $I'$, and continue in a vertical, right-bulging square-wave pattern until you hit the boundary of $W$ at a vertex $v$.
The vertex $v$ is a member of exactly two top bricks of $W$. Let $B$ be the left brick and let $I$ be its peg interval. Then $v$ is an end vertex of $I$.
By assumption $\Omega$ contains a vertex $p$ in the interior of $I$. Continue the path from $v$ to $p$ along $I$ to complete $R_{I'}$.

It is not hard to check that these constructions give the desired paths and that the four corner brick paths are mutually disjoint as required.
\end{enumerate}
\end{proof}

\appendix

\section{A Counterexample to the Flat Wall Theorem (5.2 in \cite{NewProof})}\label{CounterExample}

The Flat Wall Theorem as stated in \cite{NewProof} says:

\begin{thm52}
Let $r, t \ge 1$ be integers, let $r$ be even, let $R = 49152 t^{24} (40t^2 + r)$, let $G$ be a graph, and let $W$ be an $R$-wall in $G$. Then either $G$ has a model of a $K_t$ minor grasped by $W$,
or there exist a set $A \subseteq V(G)$ of size at most $12288t^{24}$ and an $r$-subwall $W'$ of $W$ such that $V(W') \cap A = \emptyset$ and $W'$ is a flat wall in $G \setminus A$.
\end{thm52}

This theorem fails because of a definition of flatness that is too strict\footnote{Robertson and Seymour use a much looser definition.}. Flatness is defined in \cite{NewProof} as follows:

\begin{defnFlat}
Let $G$ be a graph, and let $W$ be a wall in $G$ with an outer cycle $D$.
Let us assume that there exists a separation $(A, B)$ such that $A \cap B \subseteq V(D)$, $V(W) \subseteq B$, and there is a choice of pegs of $W$ such that every peg belongs to $A$.
If some $A \cap B$-reduction of $G[B]$ can be drawn in a disk with the vertices of $A \cap B$ drawn on the boundary of the disk in the order determined by $D$, then we say that the wall $W$ is {\em flat} in $G$.
\end{defnFlat}

The {\em choice of pegs} requirement in the definition simply means that for every top or bottom brick in $W$, at least one degree-2 vertex of $D$ along that brick must be in $A$;
for every right or left brick at least two degree-2 vertices of $D$ along the brick must be in $A$; two such vertices are in $A$ for recessed corner bricks; and finally three such vertices are in $A$ for bulging corner bricks.
See Figure \ref{fig:bricksandpegintervals} for reference.

In light of Lemmas 1.3 and 1.4 in \cite{NewProof}, the definition can be rephrased as follows (compare with Definition \ref{def:flatwall}).

\begin{defnFlat}
  Let $G$ be a graph and $W \subset G$ a wall with boundary $D$. We say that $W$ is {\em flat} in $G$ if there is a separation $(A, B)$ of $G$ such that
  \begin{enumerate}
  \item $V(W) \subset V(B)$
  \item $A \cap B \subset V(D)$
  \item $A \cap B$ contains a choice of pegs of $W$.
  \item Endow $A \cap B$ with a circular order induced from $D$. Then $(G[B], A \cap B)$ is a rural society.
  \end{enumerate}
\end{defnFlat}

The root cause of the failure of Theorem 5.2 is that $G[B]$ can include arbitrary arrangements of edges between vertices of $D$ that can prevent the society $(G[B], A \cap B)$ from being rural as required.

The counterexample to 5.2 is essentially a very large wall with some additional vertices and edges added to each brick, including a pair of crossing edges with ends along the horizontal bottom edge of the brick.
We call these bricks {\em full bricks} (see Figure \ref{fig:CeFullBrick}). We will refer to a graph built by layering full bricks in an $R$-wall-like configuration as an {\em $R$-counterwall}.
We must take care when we layer these bricks - when a brick $A$ is layered over the top left of brick $B$, their shared horizontal path is determined by $A$ and not by $B$.
If $A$ is layered on over the top right of $B$, then the shared horizontal path is just a single edge.

Write $G_R$ for an $R$-counterwall built out of full bricks.
The {\em wall of $G_R$} is the wall obtained from $G_R$ by removing all the diagonal edges of type $\omega \alpha$, $\omega \beta$, $\omega \gamma$ and $\omega \delta$ and all the curved edges of type $\alpha \gamma$ and $\beta \delta$.
A {\em sub-counterwall} of $G_R$ is a union of full bricks of $G_R$ that forms a counterwall.

\begin{figure}[htb]
\centering
\begin{subfigure}{.32\linewidth}
    \centering
    \scalebox{0.42} {
    \begin{tikzpicture}
    \draw (0, 0) -- (10, 0);
    \draw (0, 5) -- (10, 5);
    \draw (0, 0) -- (0, 5);
    \draw (10, 0) -- (10, 5);
  
    \draw (10, 5) -- (6, 0);
    \draw (10, 5) -- (7, 0);
    \draw (10, 5) -- (8, 0);
    \draw (10, 5) -- (9, 0);
  
    \draw (6, 0) to[out = -80, in = -100] (8, 0);
    \draw (7, 0) to[out = -80, in = -100] (9, 0);
  
    \filldraw [black] (0, 0) circle (3pt);
    \filldraw [black] (5, 0) circle (3pt);
    \filldraw [black] (6, 0) circle (3pt);
    \node at (6, 0.4) {$\alpha$};
    \filldraw [black] (7, 0) circle (3pt);
    \node at (7, 0.4) {$\beta$};
    \filldraw [black] (8, 0) circle (3pt);
    \node at (8, 0.4) {$\gamma$};
    \filldraw [black] (9, 0) circle (3pt);
    \node at (9, 0.4) {$\delta$};
    \filldraw [black] (10, 0) circle (3pt);
    \filldraw [black] (0, 5) circle (3pt);
    \filldraw [black] (5, 5) circle (3pt);
    \filldraw [black] (10, 5) circle (3pt);
    \node at (10, 5.4) {$\omega$};

    \end{tikzpicture}
  }
  \caption{Full brick}
  \label{fig:CeFullBrick}
  \end{subfigure}
  \hfill
  \begin{subfigure}{.32\linewidth}
    \centering
    \scalebox{0.42} {
    \begin{tikzpicture}
    \draw (0, 0) -- (10, 0);
    \draw (0, 5) -- (10, 5);
    \draw (0, 0) -- (0, 5);
    \draw (10, 0) -- (10, 5);
  
    \draw (10, 5) -- (6, 0);
    \draw (10, 5) -- (8, 0);
    \draw (10, 5) -- (9, 0);
  
    \draw (6, 0) to[out = 55, in = 125] (8, 0);
    \draw (7, 0) to[out = -80, in = -100] (9, 0);
  
    \filldraw [black] (0, 0) circle (3pt);
    \filldraw [black] (5, 0) circle (3pt);
    \filldraw [black] (6, 0) circle (3pt);
    \node at (6, 0.4) {$\alpha$};
    \filldraw [black] (7, 0) circle (3pt);
    \node at (7, 0.4) {$\beta$};
    \filldraw [black] (8, 0) circle (3pt);
    \node at (8, 0.4) {$\gamma$};
    \filldraw [black] (9, 0) circle (3pt);
    \node at (9, 0.4) {$\delta$};
    \filldraw [black] (10, 0) circle (3pt);
    \filldraw [black] (0, 5) circle (3pt);
    \filldraw [black] (5, 5) circle (3pt);
    \filldraw [black] (10, 5) circle (3pt);
    \node at (10, 5.4) {$\omega$};

    \end{tikzpicture}
  }
  \caption{Type I reduced brick}
  \end{subfigure}
  \hfill
  \begin{subfigure}{.32\linewidth}
    \centering
    \scalebox{0.42} {
    \begin{tikzpicture}
    \draw (0, 0) -- (10, 0);
    \draw (0, 5) -- (10, 5);
    \draw (0, 0) -- (0, 5);
    \draw (10, 0) -- (10, 5);
  
    \draw (10, 5) -- (6, 0);
    \draw (10, 5) -- (7, 0);
    \draw (10, 5) -- (9, 0);
  
    \draw (6, 0) to[out = -80, in = -100] (8, 0);
    \draw (7, 0) to[out = 55, in = 125] (9, 0);
  
    \filldraw [black] (0, 0) circle (3pt);
    \filldraw [black] (5, 0) circle (3pt);
    \filldraw [black] (6, 0) circle (3pt);
    \node at (6, 0.4) {$\alpha$};
    \filldraw [black] (7, 0) circle (3pt);
    \node at (7, 0.4) {$\beta$};
    \filldraw [black] (8, 0) circle (3pt);
    \node at (8, 0.4) {$\gamma$};
    \filldraw [black] (9, 0) circle (3pt);
    \node at (9, 0.4) {$\delta$};
    \filldraw [black] (10, 0) circle (3pt);
    \filldraw [black] (0, 5) circle (3pt);
    \filldraw [black] (5, 5) circle (3pt);
    \filldraw [black] (10, 5) circle (3pt);
    \node at (10, 5.4) {$\omega$};

    \end{tikzpicture}
  }
  \caption{Type II reduced brick}
  \end{subfigure}
  \caption{A full brick and two types of reduced bricks}
  \label{fig:CeBricks}
\end{figure}
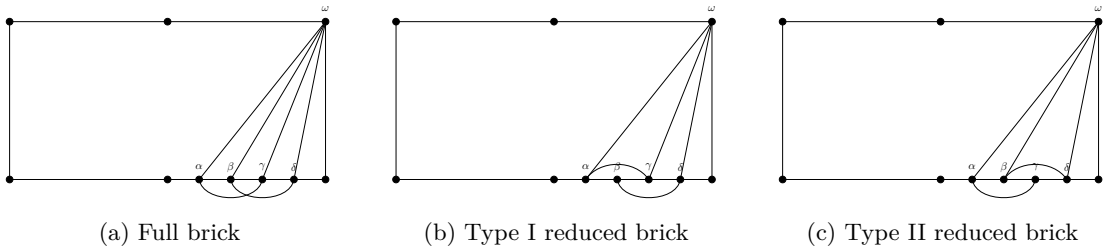

\begin{claim}
The graph $G_R$ does not possess a model of a $K_6$ minor.
\end{claim}

\begin{proof}
To prove the claim we introduce {\em reduced bricks} (see Figure \ref{fig:CeBricks}), and we consider {\em mixed counterwalls} built with a mix of full and reduced bricks. The layering rules for full bricks apply to reduced bricks as well.
We prove that mixed counterwalls do not possess a model of a $K_6$ minor. The proof proceeds by induction on the number of full bricks in the mixed counterwall.

The base case is easy. If a mixed counterwall $G$ does not contain any full bricks, then one can check that $G$ is planar by inspecting the two types of reduced bricks.
For the other cases, suppose that there is a model of a $K_6$ minor in $G$. Denote its six branch sets by $B_1, \dots, B_6$ and assume that all of them are trees. Denote its fifteen model edges by $e_{12}, e_{13}, \dots, e_{56}$.

Notice that a branch set can be a singleton $B_i = \{ v \}$ only if $v$ is a vertex of degree 5 at least. This excludes the degree 4 vertices of type $\alpha$, $\beta$, $\gamma$ and $\delta$.
The induction proceeds by choosing a full brick in $G$ and considering the following cases.
\begin{enumerate}[C{a}se 1:]
\item\label{case:UnusedEdge} One of the two edges $\omega \beta$ and $\omega \gamma$ is not used in the model - neither in some $B_i$ nor as a model edge. \\
In this case the full brick can be replaced with a reduced brick of type I or II, respectively. The $K_6$ model lifts trivially to the new mixed counterwall, contrary to the induction hypothesis.

\item\label{case:Consecutive} Two neighboring vertices $v$, $v'$ among $\alpha$, $\beta$, $\gamma$ and $\delta$ belong to the same branch set $B_i$. \\
At least one of $v$ and $v'$ is in $\{ \beta, \gamma \}$. Let's say it is $v$. Now examine the role of the vertex $\omega$.
If $\omega$ is not used in the model then we are obviously in Case \ref{case:UnusedEdge} since $\omega v$ is not used in the model either.
if $\omega$ belongs to $B_j$ with $i \ne j$, then at most one of $\omega v$ and $\omega v'$ plays the role of $e_{i j}$.
If it happens to be $\omega v$, we can modify the model by replacing $\omega v$ with $\omega v'$ in that role, and we are back to Case \ref{case:UnusedEdge}, since $\omega v$ is no longer used in the model.

So we may assume that $\omega$ belongs to $B_i$ as well. Look at the edge $\omega v$.
If $\omega v$ is not used in $B_i$ then we are back to Case \ref{case:UnusedEdge}, so assume that $\omega v \in V(B_i)$.
Removing $\omega v$ creates a disjoint union of trees $B_i \setminus \omega v = B_i^{\omega} \sqcup B_i^v$ with $\omega \in V(B_i^{\omega})$ and $v \in V(B_i^v)$.
We can go back to Case \ref{case:UnusedEdge} by removing $\omega v$ from $B_i$ and replacing it with $v v'$ (if $v' \in V(B_i^{\omega})$) or with $\omega v'$ (if $v' \in V(B_i^v)$).

\item None of the above. \\
We can assume that both $\omega \beta$ and $\omega \gamma$ are used in the model; that $\beta$ and $\gamma$ belong to two different branch sets $B_i$ and $B_j$; and that $\omega$ belongs to some branch set $B_k$.
Without loss of generality we can assume $i \ne k$. It follows that $\beta$ does not share a branch set with either $\alpha$, $\gamma$ or $\omega$.
Since $B_i$ cannot be a singleton, it must be the case that $\delta$ is a vertex of $B_i$.

Since $\omega \beta$ is used in the model by assumption, we have $e_{i k} = \omega \beta$. We can change the model by replacing $\omega \beta$ with $\omega \delta$ in the role of  $e_{i k}$, and we are back in Case \ref{case:UnusedEdge}.
\end{enumerate}
\end{proof}

\begin{claim}
Let $R > r > 5$ be integers. Let $G$ be an $R$-counterwall, $G' \subset G$ an $r$-sub-counterwall of $G$ that is disjoint from the top horizontal path of $G$, and $X \subseteq V(G)$ a vertex set such that $G'$ is $X$-free ($V(G') \cap X = \emptyset$.)
Assume that there are three $X$-free, horizontally consecutive bricks in $G$ layered completely on top of $G'$ (their bottom paths are subpaths of the boundary of $G'$.) 
Then the wall of $G'$ is not flat in $G \setminus X$ according to the definition of flatness in \cite{NewProof}.
\end{claim}

\begin{figure}[htb]
    \centering
    \scalebox{0.25} {
    \begin{tikzpicture}
      \def\y{0}
      \foreach \x in {\y, 10 + \y, 20 + \y, 30 + \y} {
        \draw (\x, \y) -- (10 + \x, \y);
        \draw (\x, 5 + \y) -- (10 + \x, 5 + \y);
        \draw (\x, \y) -- (\x, 5 + \y);
        \draw (10 + \x, \y) -- (10 + \x, 5 + \y);

        \draw (10 + \x, 5 + \y) -- (6 + \x, \y);
        \draw (10 + \x, 5 + \y) -- (7 + \x, \y);
        \draw (10 + \x, 5 + \y) -- (8 + \x, \y);
        \draw (10 + \x, 5 + \y) -- (9 + \x, \y);
  
        \draw (6 + \x, \y) to[out = -80, in = -100] (8 + \x, \y);
        \draw (7 + \x, \y) to[out = -80, in = -100] (9 + \x, \y);
  
        \filldraw [black] (\x, \y) circle (3pt);
        \filldraw [black] (5 + \x, \y) circle (3pt);
        \filldraw [black] (6 + \x, \y) circle (3pt);
        \filldraw [black] (7 + \x, \y) circle (3pt);
        \filldraw [black] (8 + \x, \y) circle (3pt);
        \filldraw [black] (9 + \x, \y) circle (3pt);
        \filldraw [black] (10 + \x, \y) circle (3pt);
        \filldraw [black] (\x, 5 + \y) circle (3pt);
        \filldraw [black] (5 + \x, 5 + \y) circle (3pt);
        \filldraw [black] (10 + \x, 5 + \y) circle (3pt);
      }
      \def\y{5}
      \foreach \x in {10 + \y, 20 + \y} {
        \draw (\x, \y) -- (10 + \x, \y);
        \draw [dashed] (\x, 5 + \y) -- (10 + \x, 5 + \y);
        \draw [dashed] (\x, \y) -- (\x, 5 + \y);
        \draw [dashed] (10 + \x, \y) -- (10 + \x, 5 + \y);

        \draw [dashed] (10 + \x, 5 + \y) -- (6 + \x, \y);
        \draw [dashed] (10 + \x, 5 + \y) -- (7 + \x, \y);
        \draw [dashed] (10 + \x, 5 + \y) -- (8 + \x, \y);
        \draw [dashed] (10 + \x, 5 + \y) -- (9 + \x, \y);
  
        \draw [dashed] (6 + \x, \y) to[out = -80, in = -100] (8 + \x, \y);
        \draw [dashed] (7 + \x, \y) to[out = -80, in = -100] (9 + \x, \y);
  
        \filldraw [black] (\x, \y) circle (3pt);
        \filldraw [black] (5 + \x, \y) circle (3pt);
        \filldraw [black] (10 + \x, \y) circle (3pt);
        \filldraw [black] (\x, 5 + \y) circle (3pt);
        \filldraw [black] (5 + \x, 5 + \y) circle (3pt);
        \filldraw [black] (10 + \x, 5 + \y) circle (3pt);
      }
      \foreach \x in {\y - 10, \y, 10 + \y, 20 + \y} {
        \filldraw [black] (6 + \x, \y) circle (3pt);
        \filldraw [black] (7 + \x, \y) circle (3pt);
        \filldraw [black] (8 + \x, \y) circle (3pt);
        \filldraw [black] (9 + \x, \y) circle (3pt);
      }
      \node at (-1, 5) {\Huge $D'$};
      \node at (14, 10) {\Huge $T$};
      \node at (35, 10.6) {\Huge $\omega$};
      \filldraw [black] (4, 5) circle (10pt);
      \filldraw [black] (11, 5) circle (10pt);
      \filldraw [black] (25, 5) circle (10pt);
      \filldraw [black] (33, 5) circle (10pt);
      \end{tikzpicture}
    }
    \caption{Section of top row of sub-counterwall $G'$, with two $X$-free bricks above}
    \label{fig:CeSubCounterwall}
\end{figure}
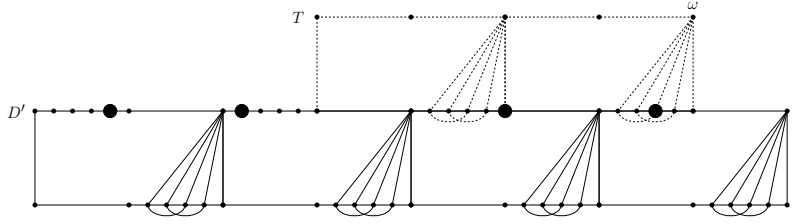

\begin{proof}
Let $W'$ be the wall of $G'$ and let $D'$ be its boundary.
Figure \ref{fig:CeSubCounterwall} shows a section of the top row of bricks of $G'$ with the rightmost two of the three $X$-free bricks above shown with dashed lines.

The section of $D'$ in the figure is the middle horizontal line and is marked with "$D'$", the dashed horizontal path above that is marked with "$T$",
and an example of a choice of pegs in the depicted section of $D'$ is marked with enlarged black circles.
The argument below is independent of any particular choice of pegs.

Assume that $W'$ is flat in $G \setminus X$. Then there is a separation $(A, B)$ of $G \setminus X$ such that:
\begin{itemize}
\item $V(W') \subseteq V(B)$
\item $A \cap B \subseteq V(D')$
\item $A \cap B$ contains a choice of pegs for $W'$.
\item When endowed with the circular order induced from $D'$, the society $(G[B], A \cap B)$ is rural.
\end{itemize}

The vertices along the path $T$ cannot belong to $A \cap B$ because $T$ is disjoint from $D'$.
Therefore, since $T \subset G \setminus X$, each vertex along $T$ must belong to $A \setminus B$ or to $B \setminus A$.
Since $(A, B)$ is a separation and $T$ is connected, either all the vertices of $T$ belong to $A \setminus B$ or all of them belong to $B \setminus A$.
\begin{enumerate}[C{a}se 1:]
\item All the vertices of $T$ belong to $A \setminus B$, and therefore $\omega$ belongs to $A \setminus B$. \\
Since the vertices $\alpha$, $\beta$, $\gamma$, $\delta$ in the top right brick belong to $B$ by assumption\footnote{labels not depicted in Figure \ref{fig:CeSubCounterwall}, see Figure \ref{fig:CeFullBrick}},
the edges $\omega \alpha, \dots, \omega \delta$ force these vertices to belong to $A \cap B$,
which is impossible since the society $(G[B], A \cap B)$ is assumed to be rural and yet it has a cross $\alpha \gamma, \beta \delta$.
\item All the vertices of $T$ belong to $B \setminus A$. \\
In this case as well there is a cross in $(G[B], A \cap B)$ as illustrated in Figure \ref{fig:CeBCross}\footnote{Technically, it is a cross only if the depicted peg choices are the rightmost choices in each brick.}.
Notice that a similar cross exists for any choice of pegs.
\end{enumerate}
\end{proof}

\begin{figure}[htb]
    \centering
    \scalebox{0.25} {
    \begin{tikzpicture}
      \def\y{0}
      \foreach \x in {\y, 10 + \y, 20 + \y, 30 + \y} {
        \draw (\x, \y) -- (10 + \x, \y);
        \draw (\x, 5 + \y) -- (10 + \x, 5 + \y);
        \draw (\x, \y) -- (\x, 5 + \y);
        \draw (10 + \x, \y) -- (10 + \x, 5 + \y);

        \draw (10 + \x, 5 + \y) -- (6 + \x, \y);
        \draw (10 + \x, 5 + \y) -- (7 + \x, \y);
        \draw (10 + \x, 5 + \y) -- (8 + \x, \y);
        \draw (10 + \x, 5 + \y) -- (9 + \x, \y);
  
        \draw (6 + \x, \y) to[out = -80, in = -100] (8 + \x, \y);
        \draw (7 + \x, \y) to[out = -80, in = -100] (9 + \x, \y);
  
        \filldraw [black] (\x, \y) circle (3pt);
        \filldraw [black] (5 + \x, \y) circle (3pt);
        \filldraw [black] (6 + \x, \y) circle (3pt);
        \filldraw [black] (7 + \x, \y) circle (3pt);
        \filldraw [black] (8 + \x, \y) circle (3pt);
        \filldraw [black] (9 + \x, \y) circle (3pt);
        \filldraw [black] (10 + \x, \y) circle (3pt);
        \filldraw [black] (\x, 5 + \y) circle (3pt);
        \filldraw [black] (5 + \x, 5 + \y) circle (3pt);
        \filldraw [black] (10 + \x, 5 + \y) circle (3pt);
      }
      \def\y{5}
      \foreach \x in {10 + \y, 20 + \y} {
        \draw (\x, \y) -- (10 + \x, \y);
        \draw [line width = 8pt] (\x, 5 + \y) -- (10 + \x, 5 + \y);
        \draw [dashed] (\x, \y) -- (\x, 5 + \y);
        \draw [dashed] (10 + \x, \y) -- (10 + \x, 5 + \y);

        \draw [dashed] (10 + \x, 5 + \y) -- (6 + \x, \y);
        \draw [dashed] (10 + \x, 5 + \y) -- (7 + \x, \y);
        \draw [dashed] (10 + \x, 5 + \y) -- (8 + \x, \y);
        \draw [dashed] (10 + \x, 5 + \y) -- (9 + \x, \y);
  
        \draw [dashed] (6 + \x, \y) to[out = -80, in = -100] (8 + \x, \y);
        \draw [dashed] (7 + \x, \y) to[out = -80, in = -100] (9 + \x, \y);
  
        \filldraw [black] (\x, \y) circle (3pt);
        \filldraw [black] (5 + \x, \y) circle (3pt);
        \filldraw [black] (10 + \x, \y) circle (3pt);
        \filldraw [black] (\x, 5 + \y) circle (3pt);
        \filldraw [black] (5 + \x, 5 + \y) circle (3pt);
        \filldraw [black] (10 + \x, 5 + \y) circle (3pt);
      }
      \foreach \x in {\y - 10, \y, 10 + \y, 20 + \y} {
        \filldraw [black] (6 + \x, \y) circle (3pt);
        \filldraw [black] (7 + \x, \y) circle (3pt);
        \filldraw [black] (8 + \x, \y) circle (3pt);
        \filldraw [black] (9 + \x, \y) circle (3pt);
      }
      \node at (-1, 5) {\Huge $D'$};
      \node at (14, 10) {\Huge $T$};
      \node at (35, 10.6) {\Huge $\omega$};

        \draw [line width = 8pt] (10 + \y, 5) -- (10 + \y, 5 + \y);
        \draw [line width = 8pt] (30 + \y, 5) -- (30 + \y, 5 + \y);
        \draw [line width = 8pt] (10, 0) -- (5 + \y, \y);
        \draw [line width = 8pt] (30, 0) -- (25 + \y, \y);
        \draw [line width = 8pt] (10, 0) -- (30, 0);
        \draw [line width = 8pt] (4, 5) -- (10, 5);
        \draw [line width = 8pt] (11, 5) -- (15, 5);
        \draw [line width = 8pt] (25, 5) -- (30, 5);
        \draw [line width = 8pt] (33, 5) -- (35, 5);

      \filldraw [black] (4, 5) circle (10pt);
      \filldraw [black] (11, 5) circle (10pt);
      \filldraw [black] (25, 5) circle (10pt);
      \filldraw [black] (33, 5) circle (10pt);
      \end{tikzpicture}
    }
    \caption{A cross along the peg choices of the boundary of $W'$}
    \label{fig:CeBCross}
\end{figure}
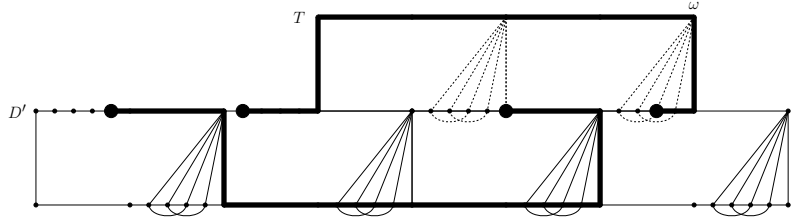

\begin{claim}
The Flat Wall Theorem is not correct as stated in Theorem 5.2 of \cite{NewProof}.
\end{claim}

\begin{proof}
Let $r, t$ be integers with $r$ even, $t \ge 6$ and $r > 4 + 36864 t^{24}$. let $R = 49152 t^{24} (40t^2 + r)$.
Let $G$ be an $R$-sub-counterwall of $G_{R+1}$ that is disjoint from the top horizontal path of $G_{R+1}$.
Let $W$ be the wall of $G$.
According to 5.2, either $G_{R+1}$ has a model of a $K_t$ minor, or
there exist a set $A \subseteq V(G_{R+1})$ of size at most $12288t^{24}$ and an $r$-subwall $W'$ of $W$ such that $V(W') \cap A = \emptyset$ and $W'$ is a flat wall in $G_{R+1} \setminus A$.

Since $t \ge 6$, $G_{R+1}$ does not have a model of a $K_t$ minor, as we have shown, and therefore according to 5.2, $A$ and $W'$ exist as specified above.
Since $W'$ is a subwall of the wall of $G$, $W'$ is the wall of a sub-counterwall $G'$ of $G$.
The top row of $G'$ contains at least $4 + 36864 t^{24} = 1 + 3(1 + 12288 t^{24})$ bricks.
By our assumption there is a row of bricks in $G_{R+1}$ directly above that row, with a consecutive series of at least $3(1 + 12288 t^{24})$ bricks that are layered completely on top of $G'$.
By dividing that series into $\ge 1 + 12288 t^{24}$ blocks of $3$ consecutive bricks,
we can conclude that by the pigeon hole principle there are 3 consecutive $X$-free bricks that are layered completlely on top of $G'$,
and therefore as we have shown $W'$ is not flat in $G_{R+1}$, contrary to the claim of 5.2.
\end{proof}

\end{document}